\documentclass[11pt]{article}
\usepackage{bbm}
\usepackage{mathrsfs}
\usepackage{amsfonts}
\usepackage{amssymb}
\usepackage{amssymb,amsmath,amsthm, amsfonts}
\usepackage{indentfirst}
\usepackage{titlesec}
\usepackage{geometry}
\usepackage{color}

\usepackage{enumerate}

\makeatletter
\@addtoreset{equation}{section}
\newtheorem{theorem}{Theorem}[section]
\newtheorem{proposition}[theorem]{Proposition}
\newtheorem{corollary}[theorem]{Corollary}
\newtheorem{lemma}[theorem]{Lemma}
\newtheorem{remark}[theorem]{Remark}
\newtheorem{example}[theorem]{Example}

\makeatother

\allowdisplaybreaks[4]
\titleformat*{\section}{\large\bfseries}
\titleformat*{\subsection}{\normalsize\bfseries}
\begin{document}

\title{A stochastic maximum principle of mean-field type with monotonicity conditions\thanks{The work is supported by the NSF of P.R. China (NOs. 12031009, 11871037), National Key R and D Program of China (NO. 2018YFA0703900), and NSFC-RS (No. 11661130148; NA150344).}}

\author{Bowen HE$^{1}$,\,\,Juan LI$^{1,2,\dag}$,\,\,Zhanxin LI$^{1,}$\footnote{Corresponding authors.}\\
 {$^1$\small School of Mathematics and Statistics, Shandong University, Weihai, Weihai 264209, P.~R.~China.}\\
 {$^2$\small  Research Center for Mathematics and Interdisciplinary Sciences, Shandong University,}\\
	{\small Qingdao 266237, P. R. China.}\\
 {\small{\it E-mails: hbw@mail.sdu.edu.cn,\,\ juanli@sdu.edu.cn,\,\  zhanxinli@mail.sdu.edu.cn.}}	}
 \date{March 15, 2025}
\renewcommand{\thefootnote}{\fnsymbol{footnote}}
\maketitle
\noindent\textbf{Abstract.}
The objective of this paper is to weaken the Lipschitz condition to a monotonicity condition and to study the corresponding Pontryagin  stochastic maximum principle (SMP) for a mean-field optimal control problem under monotonicity conditions.
The dynamics of the controlled state process is governed by a mean-field stochastic differential equation (SDE) whose coefficients depend not only on the control, the controlled state process itself but also on its law, and in particular, these coefficients satisfy the monotonicity condition with respect to both the controlled state process and its distribution. The associated cost functional is also of mean-field type. Under the assumption of a convex control domain we derive the SMP, which provides a necessary optimality condition for control processes. Under additional convexity assumptions on the Hamiltonian, we further prove that this necessary condition is also a sufficient one. To achieve this, we first address the challenges related to the existence and the uniqueness of solutions for mean-field backward stochastic differential equations and mean-field SDEs whose coefficients satisfy monotonicity conditions with respect to both the solution as well as its distribution. On the other hand we also construct several illustrative examples demonstrating the generality of our results compared to existing literature.

\vskip 0.5cm


\vskip 1cm

\section{Introduction}
Nonlinear backward stochastic differential equations (BSDEs) were first introduced by Pardoux and Peng \cite{ref41}, who established an existence and uniqueness result
of a solution of the following BSDE (1.1) under the Lipschitz assumption of the generator $g$:
\begin{equation}
y_t=\xi+\int_t^T g(s, y_s, z_s) d s-\int_t^T z_s \mathrm{d} W_s, \quad t \in[0, T],
\end{equation}
where the terminal $\xi$ is a square integral random variable, the generator of the BSDE (1.1) is a function $g(\omega, t, y, z): \Omega \times[0,T]\times\mathbb{R} \times \mathbb{R}^d \rightarrow \mathbb{R}$, which is progressively measurable for each $(y, z)$, and $W$ is a $d$-dimensional standard Brownian motion. The solution $(y_{.}, z.)$ is a pair of square integrable, adapted processes. The triple $(\xi, T, g)$ is called the parameters of the BSDE (1.1).

Since the pioneering work \cite{ref41}, many efforts have been done
to weaken the Lipschitz assumptions on the generator $g$. In particular, under the conditions that $g$
is continuous and monotonic with a general growth in $y$,  Lipschitz continuous in $z$, and $(g(t, 0, 0))_{t\in[0,T]}$ is square
integrable, Pardoux \cite{1999} proved the existence and the uniqueness of the solution to  BSDE (1.1). Furthermore, Briand et al. \cite{2007} proved the existence
of the solution to  BSDE (1.1) when the above Lipschitz continuity conditions is replaced by the continuity and linear growth
conditions. Moreover,  other generalizations can be found, for instance, in \cite{2010,2008}.

On the other hand, the modern control theory can refer to three  important achievements: Pontryagin's maximum principle, Bellman's  dynamic programming principle, and Kalman's filtering theory.  Many of the seminal works on the stochastic maximum principle are due to Kushner \cite{ref33,ref34}. Since then, huge number of works have been written on this subject, in particular, by Bensoussan \cite{ref5}, Bismut \cite{ref6}, Cadenillas, Karatzas \cite{ref17}, Elliott \cite{ref25}, Haussmann \cite{ref27} and so on.

Mean-field stochastic differential equations, also known as McKean-Vlasov equations, were initially discussed by Kac \cite{ref30,ref31}
as a stochastic ''toy model" for the Vlasov kinetic equation of plasma. Since then, these equations have attracted a lot of researchers in a variety of fields, such as physics, economics and mathematics. Lasry and  Lions \cite{ref35} studied the mean-field model in economics, finance and game theory for the first time in 2007. Motivated by these studies,  Buckdahn, Djehiche, Li and Peng \cite{ref11} introduced in 2009 a new type of BSDEs, the so-called mean-field BSDEs, by a purely stochastic approach. At the same time, Buckdahn, Li and Peng \cite{ref12} gave the corresponding comparison theorem and studied the generalized Feynman-Kac formula for the related partial differential equations. Li, Liang and Zhang \cite{JMAA} studied a class of mean-field BSDEs of the form
\begin{equation}\label{eqintro}
Y_t=\xi+\int_t^Tf(s,Y_s,Z_s,\mathbb{P}_{Y_s})ds-\int_t^TZ_sdW_s,\ \ 0\le t\le T,
\end{equation}
and they proved the existence of the minimum solution of the above equation  under continuity and linear growth conditions and gave the corresponding general comparison theorem first time.

Buckdahn, Djehiche and Li \cite{ref13} in 2011 studied the stochastic maximum principle for a mean field optimal control problem under Lipschitz condition, in which the coefficients of stochastic differential equations depend not only on the control and the state variable, but also on its expectation. More precisely, the  form of the control problem is as follows
\begin{equation}\label{eq 1.1}\tag{1.2}
\left\{
             \begin{array}{l}
           dx(t)=b(t,x(t),E[x(t)],u(t))dt+\sigma(t,x(t),E[x(t)],u(t))dW(t),\ t\in[0,T],\\
x(0)=x_{0} ,
             \end{array}
\right.
\end{equation}
with the cost functional
\begin{equation}\label{eq 1.2}\tag{1.3}
J(u(\cdot))=E\left[\int_0^Tf(t,x(t),E[x(t)],u(t))dt+h(x(T),E[x(T)])\right],
\end{equation}
where $f$ and $h$ are  given functions. However, here $J$ is, in general, nonlinear with respect to (w.r.t.) the expectation of $x(t)$, which leads to time inconsistency of the control problem. When the control field $\mathcal{U}$ is convex, Li \cite{ref36}, Andersson and Djehiche \cite{ref1} obtained the necessary and the sufficient conditions of the optimal control in their control problems, respectively. Since then, there are many works for such type of mean-field control problems, for example,  Meyer-Brandis, {\O}ksendal and Zhou \cite{ref39} employed Malliavin calculus to investigate the maximum principle for mean-field control problems under partial information, where the coefficients of the cost functional depend on the expectation of the state process, under the assumption of continuously differentiable coefficients.

Since Lions \cite{ref58} (see also  Cardaliaguet \cite{ref21})  gave the definition and a characterization of the derivatives of functions over a space of probability measures,  Buckdahn, Li and Ma \cite{ref15}  studied in 2016 the stochastic maximum principle of the following general controlled equations under Lipschitz conditions,
\begin{equation}\label{eq 1.3}\tag{1.4}
\left\{
             \begin{array}{l}
           dX^u_t=b(t,X^u_t,\mathbb{P}_{X^u_t},u_t)dt+\sigma(t,X^u_t,\mathbb{P}_{X^u_t},u_t)dW_t,\ t\in[0,T],\\
X^u_0=x_{0} ,
             \end{array}
\right.
\end{equation}
where the coefficients of the equation depend not only on the state variables, but also on the law of the state variables. The cost functional is given by
\begin{equation}\label{eq 1.4}\tag{1.5}
J(u)=E\left[\int_0^Tf(t,X^u_t,\mathbb{P}_{X^u_t},u_t)dt+h(X^u_T,\mathbb{P}_{X^u_T})\right].
\end{equation}
The stochastic maximum principle in this case is a further extension of the work by Buckdahn et al. in \cite{ref13}. They still studied the stochastic maximum principle of the corresponding control problem under Lipschitz condition.

  Orrieri \cite{ref40} in 2015 investigated the stochastic maximum principle where the dynamics is a classical SDE (without the mean-field term) whose drift coefficient satisfies the  monotonicity condition w.r.t. the state process and diffusion coefficient satisfies the  Lipschitz condition w.r.t. the state process. Boufoussi and  Mouchtabih \cite{BM23} studied the existence and the uniqueness of the
solutions to the mean-field BSDEs whose generators depend on the solution $(Y,Z)$ and the law of $Y$, with the assumption that generators are locally monotone in $y$, locally Lipschitz w.r.t. $z$ and law's variable. Chen, Nie and Wang \cite{CNW23}  studied a stochastic maximum principle where  the state dynamics is given
 by controlled mean-field FBSDEs, whose coefficients are monotone in $x$ and $y$, but Lipschitz w.r.t. $z$ and the joint distribution of the state processes
 and the control.

Inspired by  above works, our paper generalizes the stochastic control problem (1.4)-(1.5). We assume that the  coefficients of  equation (1.4) satisfy  conditions (H1)-(H3). 
We give some necessary estimates under these conditions, and we prove the stochastic maximum principle of Pontryagin type. In addition, different from the Lipschitz case, here the adjoint equation  is a mean-field BSDE under the conditions of monotonicity, so first we need to prove the existence and the  uniqueness of its solution. We adapt the idea of Darling and Pardoux \cite{Pardoux97} to prove the existence of the solution. But unlike their method in \cite{Pardoux97} because now our generator $f$ also depends on the measure and satisfies a monotonicity condition w.r.t. the measure, we have to construct a new type of sequence of Lipschitz functions to approximate the function $f$. For this, we introduce a new form of convolution function. We further provide several illustrating examples to show the extent to which we generalize the results in the literature.

Our paper is organized as follows: In Section 2 and 3 we recall some notions and prove some results of mean-field BSDEs and mean-field SDEs under the condition of monotonicity.
Finally, in Section 4 we give the formulation of our optimal control problem and study the stochastic maximum principle of Pontryagin type.

\section{Mean field BSDEs  with monotonicity conditions}

We consider a filtered complete probability space $(\Omega,\mathcal{F},\mathbb{F},\mathbb{P})$. Let $W=\{W_s\}_{s\geq 0}$ be a one-dimensional standard Brownian motion defined on this probability space and $T > 0 $ a given time horizon. We assume that there exists a sub-$\sigma$-field $\mathcal{F}_0\subset\mathcal{F}$ that is independent of  $\mathbb{F}^W$, the filtration generated by $W$, and ''rich enough" in the sense that for every $\mu\in \mathscr{P}_2(\mathbb{R}^d)$, there is a random variable $\vartheta\in L^2(\mathcal{F}_0;\mathbb{R}^d)$ such that $\mathbb{P}_{\vartheta}=\mu$. We denote $\mathbb{F}=\{\mathcal{F}_s,0\leq s\leq T\}\vee\mathcal{F}_0$  which is augmented by all the $\mathbb{P}$-null sets. For a generic Euclidean space $\mathbb{X}$, its inner product is denoted by $(\cdot,\cdot)$, its norm by $|\cdot|$, and its Borel $\sigma$-field by $\mathcal{B(\mathbb{X})}$.

We introduce the following spaces: for any sub-$\sigma$-field $\mathcal{G}\subseteq\mathcal{F}$ and $1\le p< \infty$,

\noindent$\bullet\ {L^p}(\mathcal{G};\mathbb{X})$\ :=\Big \{$\xi$:\ $\mathbb{X}$-valued, $\mathcal{G}$-measurable random variable, ${{\|\xi\|}_p}:=\mathbb{E}\big[{\|\xi\|}^p\big]^{\frac{1}{p}}<\infty$\Big\}. In particular, ${L^2}(\mathcal{G};{\mathbb{R}}^d)$ is a\ Hilbert space, where the inner product is \ $(\xi,\eta)_2:= E[\xi\cdot\eta]$,\ $\xi,\eta\in{L^2}(\mathcal{G};{\mathbb{R}}^d)$,\ and the norm is ${{\|\xi\|}_2}=\sqrt{(\xi,\xi)_2}$.

\noindent$\bullet\ L^p_{\mathbb{F}}(0,T;\mathbb{X})$:=\Big\{$\eta$:\ $\mathbb{X}$-valued\ $\mathbb{F}$-adapted process on $[0,T]$,
${{\|\eta\|}_{p,T}}:= E\big[\int_0^T{|\eta_t|}^p dt\big]^{\frac{1}{p}}<\infty$\Big\}.

\noindent$\bullet$ $\mathcal{S}^p(0,T;\mathbb{X}):=\Big\{x(\cdot):  \mathbb{X}\text{-valued}\  \mathbb{F}\text{-adapted RCLL process} \text{\ with}\ E\big[\sup\limits_{t\in[0,T]}|x(t)|^pdt\big]<\infty\Big\}$.



\noindent$\bullet\ \ \mathscr{P}_2(\mathbb{X})$\ :=\ \{$\mu$ : probability measure on \ $(\mathbb{X},\mathcal{B}(\mathbb{X}))$\  with $\int_{\mathbb{X}}|x|^2\mu(dx)<\infty$\ \}.  The space $\mathscr{P}_2(\mathbb{R}^d)$ is equipped with the  2-Wasserstein metric: For\ $\mu,\nu\in\mathscr{P}_2(\mathbb{R}^d)$,
\begin{equation}\label{eq 2.9}
\begin{split}
W_2(\mu,\nu):= \inf\bigg\{ \left[\int_{\mathbb{R}^{2d}}|x-y|^2\rho(dxdy)\right]^{\frac{1}{2}}:\rho\in\mathscr{P}_2(\mathbb{R}^{2d}),\\ \rho(\cdot\times\mathbb{R}^d)=\mu,\ \rho(\mathbb{R}^d\times\cdot)=\nu\bigg\}.
\end{split}
\end{equation}


\indent For a given set of parameters $\alpha$ and $\beta$, $C_{\alpha}$ or $C_{\alpha,\beta}$ will denote a positive
constant only depending on these parameters and may change from line to line. Below we give some results that will be used later.
\begin{theorem}\label{le2.4}
Under the assumptions\ $\rm(A1)-(A4)$ stated below, the following mean-field backward stochastic differential equation
\begin{equation}\label{eq 2.7}
Y_t=\xi+\int_t^Tf(s,Y_s,Z_s,\mathbb{P}_{(Y_s,Z_s)})ds-\int_t^TZ_sdW_s,\ \ 0\le t\le T,
\end{equation}
admits a unique $\mathbb{F}$-adapted solution  $(Y_t,Z_t)_{t\in[0,T]}\in\mathcal{S}^2(0,T;\mathbb{R})\times L^2_{\mathbb{F}}(0,T;\mathbb{R})$.
\begin{itemize}
\item[$\rm(A1)$] $f(s,\omega,y,z,\mu):[0,T]\times\Omega\times\mathbb{R}\times\mathbb{R}\times\mathscr{P}_2(\mathbb{R}^2)\rightarrow{\mathbb{R}}$ is $\mathbb{F}$-adapted for every $(y,z,\mu)\in \mathbb{R}\times\mathbb{R}\times\mathscr{P}_2(\mathbb{R}^2)$, and
    there exists a constant\ $K\ge 0$,\ such that\ $\mathbb{P}$-a.s.,\ for all\ $t\in[0,T],\ y,\ z\in\mathbb{R},\ Y,\ Z\in L^2(\mathcal{F};\mathbb{R})$,\
\begin{equation}\label{A11}
|f(s,y,z,\mathbb{P}_{(Y,Z)})|\leq K\big(1+|y|+\big(E[|Y|^2]\big)^{\frac{1}{2}}\big)+|f(s,0,z,\mathbb{P}_{(0,Z)})|.
\end{equation}

\item[$\rm(A2)$]
With the notation
$$
\widetilde{f}(s,\omega,y,z,m,\mathbb{P}_{(\mathring{\xi},\eta)}):= f(s,\omega,y,z,\mathbb{P}_{(\mathring{\xi}+m,\eta)}),
$$
$\widetilde{f}$ satisfies the following conditions: There exists a constant\ $\alpha_1>0$, 
such that
\begin{equation}\label{A12}
\begin{split}
&\big|\widetilde{f}(s,\omega,y,z,m,\mathbb{P}_{(\mathring{\xi},\eta)})-\widetilde{f}(s,\omega,y,z',m,\mathbb{P}_{(\mathring{\xi}',\eta')})\big|\le \alpha_1\big(\vert z-z' \vert+W_2(\mathbb{P}_{(\mathring{\xi},\eta)},\mathbb{P}_{(\mathring{\xi}',\eta')})\big),
\end{split}
\end{equation}
 for all\ $s\in[0,T],\ \omega\in\Omega,\ y,\ z,\ z',\ m\in\mathbb{R},\ \mathring{\xi},\ \mathring{\xi}'\in L_0^2(\mathcal{F};\mathbb{R})\big(:=\big\{\mathring{\xi}\in L^2(\mathcal{F};\mathbb{R}):E[\mathring{\xi}]=0\big\}\big),\ \eta,\ \eta'\in L^2(\mathcal{F};\mathbb{R})$.

\item[$\rm(A3)$] $\widetilde{f}$ is continuous w.r.t. $y$ and $m$, and there exist constants\ $\alpha_2,\ \alpha_3\ge 0$,\ such that for all\ $s\in[0,T],\ y,\ y',\ z,\ m\in\mathbb{R},\ \mathring{\xi}\in L_0^2(\mathcal{F};\mathbb{R}),\ \theta,\ \zeta,\ Y,\ Y',\ \eta  \in L^2(\mathcal{F};\mathbb{R})$,
\begin{equation}\label{A13}
\begin{split}
(\widetilde{f}(s,\omega,y,z,m,\mathbb{P}_{(\mathring{\xi},\eta)})-\widetilde{f}(s,\omega,y',z,m,\mathbb{P}_{(\mathring{\xi},\eta)}))(y-y') \le \alpha_2{\vert y-y' \vert}^2,
\end{split}
\end{equation}
\begin{equation}\label{A14}
\begin{split}
E\big[(\widetilde{f}(s,\theta,\zeta,E[Y],\mathbb{P}_{(\mathring{\xi},\eta)})-\widetilde{f}(s,\theta,\zeta,E[Y'],\mathbb{P}_{(\mathring{\xi},\eta)}))(Y-Y')\big] \le \alpha_3E\big[{\vert Y-Y' \vert}^2\big].
\end{split}
\end{equation}
\begin{equation}\label{A15}
\begin{aligned}
&\text{Moreover,~} (y,m)\rightarrow\widetilde{f}(s,\omega,y,\eta,m,\mathbb{P}_{(\mathring{\xi},\eta)}) \text{~is uniformly continuous on compacts, uniformly ~}\\
 &\text{w.r.t.~} \mathbb{P}_{(\mathring{\xi},\eta)},~ds\mathbb{P}(d\omega)\text{-a.e.}
\end{aligned}
\end{equation}

\item [$\rm(A4)$]$E[|\xi|^{2}]<+\infty$, $|f(s,0,0,\delta_{\mathbf{0}})|\leq\alpha_4,~dsd\mathbb{P}$-a.e., where $\xi$ is $\mathcal{F}_{T}$-measurable, and $\delta_\mathbf{0}$ is the Dirac measure at $\mathbf{0}\in\mathbb{R}^2$.
\end{itemize}
\end{theorem}
\begin{remark}\rm
Condition (A3)\eqref{A14} is satisfied in the following both important special cases:
\begin{itemize}
\item[$\rm i)$] That of $L^2$-continuity of $m\rightarrow\widetilde{f}(s,\omega,\theta,\zeta,m,\mathbb{P}_{(\mathring{\xi},\eta)})$, i.e.,
    $$
    E\big[|\widetilde{f}(s,\theta,\zeta,m,\mathbb{P}_{(\mathring{\xi},\eta)})-\widetilde{f}(s,\theta,\zeta,m',\mathbb{P}_{(\mathring{\xi},\eta)})|^2\big] \le \alpha_3^2| m-m'|^2,
    $$
    for all\ $s\in[0,T],\ m,\ m'\in\mathbb{R},\ \mathring{\xi}\in L_0^2(\mathcal{F};\mathbb{R}),\ \theta,\ \zeta,\ \eta  \in L^2(\mathcal{F};\mathbb{R})$;
\item[$\rm ii)$] And that, where $\widetilde{f}(s,\omega,y,z,m,\mathbb{P}_{(\mathring{\xi},\eta)})=\widetilde{f}_0(s,\omega,y,z,\mathbb{P}_{(\mathring{\xi},\eta)})+\widetilde{f}_1(s,m,\mathbb{P}_{(\mathring{\xi},\eta)})$, and $\widetilde{f}_1$ is such that
    $$
    (\widetilde{f}_1(s,m,\mathbb{P}_{(\mathring{\xi},\eta)})-\widetilde{f}_1(s,m',\mathbb{P}_{(\mathring{\xi},\eta)}))(m-m') \le \alpha_3| m-m'|^2,
    $$
    for all\ $s\in[0,T],\ \omega \in \Omega,\ y,\ z,\ m,\ m'\in\mathbb{R},\ \mathring{\xi}\in L_0^2(\mathcal{F};\mathbb{R}),\ \eta  \in L^2(\mathcal{F};\mathbb{R})$.
\end{itemize}
\end{remark}
\begin{proof}[Proof.]
\textit{Uniqueness.} Assume that $(Y^1,Z^1),\ (Y^2,Z^2)$ are both  solutions of BSDE (\ref{eq 2.7}), and we put $\widehat Y:=Y^1-Y^2,\ \widehat Z:=Z^1-Z^2$. By applying It\^{o}'s formula to $e^{\beta s}|\widehat{Y}_s|^2$, we have for any $\beta\in\mathbb{R}$,
\begin{equation*}
\begin{aligned}
&E[|\widehat{Y}_0|^2]+E\left[\int_0^T \beta e^{\beta s}|\widehat{Y}_s|^2ds\right]+E\left[\int_0^T  e^{\beta s}|\widehat{Z}_s|^2ds\right]\\
&=2E\left[\int_0^T  e^{\beta s}\widehat{Y}_s(f(s,Y^1_s,Z^1_s,\mathbb{P}_{(Y^1_s,Y^1_s)})-f(s,Y^2_s,Z^1_s,\mathbb{P}_{(Y^1_s,Z^1_s)}))
ds\right]\\
&\quad +2E\left[\int_0^T  e^{\beta s}\widehat{Y}_s(f(s,Y^2_s,Z^1_s,\mathbb{P}_{(Y^1_s,Z^1_s)})-f(s,Y^2_s,Z^2_s,\mathbb{P}_{(Y^2_s,Z^2_s)})) ds\right]\\
&= 2E\left[\int_0^T  e^{\beta s}\widehat{Y}_s(f(s,Y^1_s,Z^1_s,\mathbb{P}_{(Y^1_s,Y^1_s)})-f(s,Y^2_s,Z^1_s,\mathbb{P}_{(Y^1_s,Z^1_s)}))
ds\right]\\
&\quad +2E\left[\int_0^T  e^{\beta s}\widehat{Y}_s(\widetilde{f}(s,Y^2_s,Z^1_s,E[Y^1_s],\mathbb{P}_{(Y^1_s-E[Y^1_s],Z^1_s)})-\widetilde{f}(s,Y^2_s,Z^1_s,E[Y^2_s],\mathbb{P}_{(Y^1_s-E[Y^1_s],Z^1_s)})) ds\right]\\
&\quad +2E\left[\int_0^T  e^{\beta s}\widehat{Y}_s(\widetilde{f}(s,Y^2_s,Z^1_s,E[Y^2_s],\mathbb{P}_{(Y^1_s-E[Y^1_s],Z^1_s)})-\widetilde{f}(s,Y^2_s,Z^2_s,E[Y^2_s],\mathbb{P}_{(Y^2_s-E[Y^2_s],Z^2_s)})) ds\right].
\end{aligned}
\end{equation*}
Thanks to the assumptions (A2) and (A3), we obtain
\begin{equation*}
\begin{aligned}
&E\left[\int_0^T \beta e^{\beta s}|\widehat{Y}_s|^2ds\right]+E\left[\int_0^T  e^{\beta s}|\widehat{Z}_s|^2ds\right]\\
&\le 2\alpha_2 E\left[\int_0^T  e^{\beta s}|\widehat{Y}_s|^2ds\right]+2\alpha_3 E\left[\int_0^T  e^{\beta s}|\widehat{Y}_s|^2ds\right]\\
& \quad + 2\alpha_1 E\left[\int_0^T  e^{\beta s}|\widehat{Y}_s|\Big(|\widehat{Z}_s|+W_2\big(\mathbb{P}_{(Y^1_s-E[Y^1_s],Z^1_s)},\mathbb{P}_{(Y^2_s-E[Y^2_s],Z^2_s)}\big)\Big)ds\right]\\
&\leq (2\alpha_2+2\alpha_3) E\left[\int_0^T  e^{\beta s}|\widehat{Y}_s|^2ds\right]+2\alpha_1 E\left[\int_0^T  e^{\beta s}\widehat{Y}_s\Big(|\widehat{Z}_s|+\big(E\big[\big|\widehat{Y}_s-E[\widehat{Y}_s]\big|^2+|\widehat Z_s|^2\big]\big)^{\frac{1}{2}}\Big) ds\right]\\
&\leq (2\alpha_2+2\alpha_3+8\alpha_1^2+\frac{1}{2}) E\left[\int_0^T  e^{\beta s}|\widehat{Y}_s|^2ds\right]+\frac{1}{2} E\left[\int_0^T  e^{\beta s}|\widehat Z_s|^2 ds\right].
\end{aligned}
\end{equation*}
Taking $\beta=2\alpha_2+2\alpha_3+8\alpha_1^2+1$, we have
$$
E\left[\int_0^T e^{\beta s}|\widehat{Y}_s|^2ds\right]+E\left[\int_0^T  e^{\beta s}|\widehat{Z}_s|^2ds\right]\leq 0,
$$
therefore $Y_t^1=Y_t^2$, $Z_t^1=Z_t^2$, $dtd\mathbb{P}$-a.s.

\textit{Existence.}
Let us admit for a moment the following proposition.

\begin{proposition}\label{111}
Let $V\in L_\mathbb{F}^2(0,T;\mathbb{R})$  be any given stochastic process, then there exists a unique pair $(Y,Z)\in \mathcal{S}^2(0,T;\mathbb{R})\times L_\mathbb{F}^2(0,T;\mathbb{R})$ satisfying
\begin{equation}\label{MFBSDEV}Y_t=\xi+\int_t^Tf(s,Y_s,V_s,\mathbb{P}_{(Y_s,V_s)})ds-\int_t^TZ_sdW_s,\ \ 0\le t\le T.\end{equation}
\end{proposition}

Using Proposition \ref{111}, we can construct a sequence $(Y^n,Z^n)$ as follows: $(Y^0,Z^0):=(0,0)$, and, for $n\geq 1$,
$$
Y^{n+1}_t=\xi+\int_t^Tf(s,Y^{n+1}_s,Z^n_s,\mathbb{P}_{(Y^{n+1}_s,Z^{n}_s)})ds-\int_t^TZ^{n+1}_sdW_s,\ \ 0\le t\le T.
$$
Let $\Delta Y^n\equiv Y^{n+1}-Y^n$, and $\Delta Z^n\equiv Z^{n+1}-Z^n$. By applying It\^{o}'s formula to $e^{\theta t}|\Delta Y^n_t|^2$, we get for any $t\in [0,T],\ \theta\in \mathbb{R}$,
\begin{equation}
\begin{aligned}
&E\big[e^{\theta t}|\Delta Y^n_t|^2+\int_t^Te^{\theta s}\big(\theta |\Delta Y^n_s|^2+|\Delta Z^n_s|^2\big)ds\big]\\
&=2E\Big[\int_t^Te^{\theta s}\Delta Y^n_s\big(f(s,Y^{n+1}_s,Z^n_s,\mathbb{P}_{(Y^{n+1}_s,Z^{n}_s)})-f(s,Y^{n}_s,Z^{n-1}_s,\mathbb{P}_{(Y^{n}_s,Z^{n-1}_s)})\big)ds\Big]\\
&=2E\Big[\int_t^Te^{\theta s}\Delta Y^n_s\big(f(s,Y^{n+1}_s,Z^n_s,\mathbb{P}_{(Y^{n+1}_s,Z^{n}_s)})-f(s,Y^{n}_s,Z^{n}_s,\mathbb{P}_{(Y^{n+1}_s,Z^{n}_s)})\big)ds\Big]\\
&\quad +2E\Big[\int_t^Te^{\theta s}\Delta Y^n_s\Big(\widetilde{f}(s,Y^{n}_s,Z^n_s,E[Y^{n+1}_s],\mathbb{P}_{(Y^{n+1}_s-E[Y^{n+1}_s],Z^{n}_s)})\\
&\text{ } \hspace{3.7cm}-\widetilde{f}(s,Y^{n}_s,Z^n_s,E[Y^{n}_s],\mathbb{P}_{(Y^{n+1}_s-E[Y^{n+1}_s],Z^{n}_s)})\Big)ds\Big]\\
&\quad +2E\Big[\int_t^Te^{\theta s}\Delta Y^n_s\Big(\widetilde{f}(s,Y^{n}_s,Z^n_s,E[Y^{n}_s],\mathbb{P}_{(Y^{n+1}_s-E[Y^{n+1}_s],Z^{n}_s)})\\
&\text{ } \hspace{3.7cm}-\widetilde{f}(s,Y^{n}_s,Z^{n-1}_s,E[Y^n_s],\mathbb{P}_{(Y^{n}_s-E[Y^n_s],Z^{n-1}_s)})\Big)ds\Big]
\end{aligned}
\end{equation}
\begin{equation}
\begin{aligned}
&\leq (2\alpha_2+2\alpha_3)E\Big[\int_t^Te^{\theta s}|\Delta Y^n_s|^2ds\Big]\\
&\quad +2\alpha_1 E\Big[\int_t^Te^{\theta s}|\Delta Y^n_s|\Big(|\Delta Z^{n-1}_s|+W_2\big(\mathbb{P}_{(Y^{n+1}_s-E[Y^{n+1}_s],Z^{n}_s)}, \mathbb{P}_{(Y^{n}_s-E[Y^n_s],Z^{n-1}_s)}\big)\Big)ds\Big]\\
&\leq (2\alpha_2+2\alpha_3)E\Big[\int_t^Te^{\theta s}|\Delta Y^n_s|^2ds\Big]\\
&\quad +2\alpha_1 E\Big[\int_t^Te^{\theta s}|\Delta Y^n_s|\Big(|\Delta Z^{n-1}_s|+\big(E\big[\big|\Delta Y^{n}_s-E[\Delta Y^{n}_s]\big|^2+| \Delta Z^{n-1}_s|^2\big]\big)^{\frac{1}{2}}\Big)ds\Big]\\
&\leq (2\alpha_2+2\alpha_3+c\alpha_1^2+\frac{4}{c} )E\Big[\int_t^Te^{\theta s}|\Delta Y^n_s|^2ds\Big]+\frac{4}{c}E\Big[\int_t^Te^{\theta s}|\Delta Z^{n-1}_s|^2ds\Big],
\end{aligned}
\end{equation}
for any $c>0$, where the first inequality comes from assumptions (A2) and (A3). Choosing $c=8,\ \theta=2\alpha_2+2\alpha_3+8\alpha_1^2+\frac{1}{2}$, we deduce
\begin{equation}
E\big[\int_t^Te^{\theta s}|\Delta Z^n_s|^2ds\big]\leq \frac{1}{2}E\Big[\int_t^Te^{\theta s}|\Delta Z^{n-1}_s|^2ds\Big].
\end{equation}
Hence, $\{Z^n\}_{n\geq 1}$ is a Cauchy sequence in $L^2_{\mathbb{F}}(0,T;\mathbb{R})$, and tends to a limit $Z$. Choosing $c=4,\ \theta=2\alpha_2+2\alpha_3+5\alpha_1^2+1$, we also  get
\begin{equation}
E\big[\int_t^Te^{\theta s}|\Delta Y^n_s|^2ds\big]\leq \frac{1}{\alpha_1^2}E\Big[\int_t^Te^{\theta s}|\Delta Z^{n-1}_s|^2ds\Big]\rightarrow0,\ n\rightarrow\infty,
\end{equation}
and so also $\{Y^n\}_{n\geq 1}$ is a Cauchy sequence in $L^2_{\mathbb{F}}(0,T;\mathbb{R})$, and, hence, it has a limit $Y$. Following the standard BSDEs theory, we show $(Y,Z)$ solves BSDE \eqref{eq 2.7}.
\end{proof}
\begin{lemma}\label{lemmajuanji}
Let $f(\omega,t,y,z,\mu):\Omega\times[0,T]\times\mathbb{R}\times\mathbb{R}\times\mathscr{P}_2(\mathbb{R}^2)\rightarrow{\mathbb{R}}$ satisfy the assumptions (A1)-(A4). For any fixed $V\in L^2_{\mathbb{F}}(0,T;\mathbb{R})$, we define $g:[0,T]\times \Omega\times \mathbb{R}\times \mathbb{R}\times\mathcal{P}_2(\mathbb{R}^2)\rightarrow \mathbb{R}$ as $g(s,y,m,\mathbb{P}_{(\mathring{Y},V_s)}):=\widetilde{f}(s,y,V_s,m,\mathbb{P}_{(\mathring{Y},V_s)})=f(s,y,V_s,\mathbb{P}_{(\mathring{Y}+m,V_s)})$, where $\mathring{Y}=Y-E[Y]$, $Y\in L^2(\Omega,\mathcal{F}_T,\mathbb{P})$.  Let $\rho(y',y'')\in C^1_c(\mathbb{R}^2)$ be such that its compact support satisfies $supp(\rho)\subset[0,1]^2$, $\rho\geq0$, and $\int_{\mathbb{R}^2}\rho(u)du=1$.   Then the sequence of functions $g_n:[0,T]\times \Omega\times \mathbb{R}\times \mathcal{P}_2(\mathbb{R}^2)\rightarrow \mathbb{R}$,
\begin{equation}\label{juanji}
g^n(s,y,\mathbb{P}_{(Y,V_s)}):=\int_{\mathbb{R}^2}n^2g(s,y-y',E[Y]-y'',\mathbb{P}_{( \mathring{Y},V_s)})\rho(ny',ny'')dy'dy'',
\end{equation}
is well defined, and has the following properties:

$(\rm i)$ Lipschitz condition: For all $s\in[0,T],\ y,~\hat{y}\in\mathbb{R},\ Y,~\hat{Y}\in L^2(\mathcal{F})$,
$$
|g^n(s,y,\mathbb{P}_{(Y,V_s)})-g^n(s,\hat{y},\mathbb{P}_{(\hat{Y},V_s)})|\leq C_n \Big(|y-\hat{y}|+W_2(\mathbb{P}_Y,\mathbb{P}_{\hat{Y}})\Big),\quad \mathbb{P}\text{-a.s.,}
$$
where $C_n>0$ is a constant depending on $n$.

$(\rm ii)$ There exists a constant $C_{K,\alpha_4}>0$ such that, for all $s\in[0,T],\ y\in\mathbb{R},\ Y\in L^2(\mathcal{F})$,
$$ |g^n(s,y,\mathbb{P}_{(Y,V_s)})|\leq C_{K,\alpha_4}(1+|y|+|V_s|+(E[|Y|^2])^{\frac{1}{2}}+(E[|V_s|^2])^{\frac{1}{2}}),\quad \mathbb{P}\text{-a.s.}$$


$(\rm iii)$ Convergence: Let ~$y\in\mathbb{R}$ and ~$Y \in L^2(\mathcal{F})$, then $g^n(s,y,\mathbb{P}_{(Y,V_s)})\rightarrow g(s,y,E[Y],\mathbb{P}_{(\mathring{Y},V_s)})=f(s,y,V_s,\mathbb{P}_{(Y,V_s)})$, as $n\rightarrow \infty$, for all $s\in[0,T]$.
\end{lemma}
\begin{proof} (i) Since
\begin{equation*}
\begin{aligned}
&g^n(s,y,\mathbb{P}_{(Y,V_s)})=\int_{\mathbb{R}^2}n^2g(s,y-y',E[Y]-y'',\mathbb{P}_{( \mathring{Y},V_s)})\rho(ny',ny'')dy'dy''\\
=&\int_{\mathbb{R}^2}n^2g(s,x',x'',\mathbb{P}_{( \mathring{Y},V_s)})\rho\big(n(y-x'),n(E[Y]-x'')\big)dx'dx'',
\end{aligned}
\end{equation*}
it follows from the assumptions on $\rho$ and (A3)\eqref{A15}, we have
\begin{equation}\label{gny}
|\partial_y g^n(s,y,\mathbb{P}_{(Y,V_s)})|\leq C_n \int_{\mathbb{R}^2}|g(s,x',x'',\mathbb{P}_{( \mathring{Y},V_s)})|dx'dx''<\infty.
\end{equation}
On the other hand, for all $Y,~\hat{Y}\in L^2(\mathcal{F})$, by using our assumptions on $g$ and $\rho$, we obtain
\begin{equation}\label{gnY}
\begin{aligned}
&|g^n(s,y,\mathbb{P}_{(Y,V_s)})-g^n(s,y,\mathbb{P}_{(\hat{Y},V_s)})|\\
=&\Big|\int_{\mathbb{R}^2}n^2g(s,x',x'',\mathbb{P}_{( \mathring{Y},V_s)})\rho\big(n(y-x'),n(E[Y]-x'')\big)dx'dx''\\
&-\int_{\mathbb{R}^2}n^2g(s,x',x'',\mathbb{P}_{( \mathring{\hat{Y}},V_s)})\rho\big(n(y-x'),n(E[\hat{Y}]-x'')\big)dx'dx''\Big|\\
=&\Big|\int_{\mathbb{R}^2}n^2\big(g(s,x',x'',\mathbb{P}_{( \mathring{Y},V_s)})-g(s,x',x'',\mathbb{P}_{( \mathring{\hat{Y}},V_s)})\big)\rho\big(n(y-x'),n(E[Y]-x'')\big)dx'dx''\\
&+\int_{\mathbb{R}^2}n^2g(s,x',x'',\mathbb{P}_{( \mathring{\hat{Y}},V_s)})\big[\rho\big(n(y-x'),n(E[Y]-x'')\big)-\rho\big(n(y-x'),n(E[\hat{Y}]-x'')\big)\big]dx'dx''\Big|\\
\leq&CW_2(\mathbb{P}_{( \mathring{Y},V_s)},\mathbb{P}_{( \mathring{\hat{Y}},V_s)})+C_n|E[Y-\hat{Y}]|\leq C_n(E[|Y-\hat{Y}|^2])^{\frac{1}{2}}.
\end{aligned}
\end{equation}
By combining \eqref{gny}, \eqref{gnY} and taking into account the arbitrariness of $Y$ and $\hat{Y}$, we conclude
$$
|g^n(s,y,\mathbb{P}_{(Y,V_s)})-g^n(s,\hat{y},\mathbb{P}_{(\hat{Y},V_s)})|\leq C_n \Big(|y-\hat{y}|+W_2(\mathbb{P}_Y,\mathbb{P}_{\hat{Y}})\Big),\quad \mathbb{P}\text{-a.s.}
$$


(ii) From the definition of $g$, $g^n$ and the linear growth condition satisfied by $f$, we have immediately, for all $(s, y, Y) \in[0, T] \times \mathbb{R} \times L^2(\mathcal{F})$,
$$
\begin{aligned}
|g^n(s, y, \mathbb{P}_{(Y, V_s)})|& \leq K(1+|y|+(E[|Y|^2])^{\frac{1}{2}})+K\int_{\mathbb{R}^2}n^2(|y'|+|y''|)\rho(ny',ny'')dy'dy''\\
&~~~+|f(s,0,V_s,\mathbb{P}_{(0,V_s)})|\\
&\leq C_K(1+|y|+|V_s|+(E[|Y|^2])^{\frac{1}{2}}+(E[|V_s|^2])^{\frac{1}{2}})+\alpha_4,
\end{aligned}
$$
where $C_K>0$ is a constant depending on $K$.  This implies (ii).

(iii) Let  $(s, y, Y) \in[0, T] \times \mathbb{R} \times L^2(\mathcal{F})$. Thanks to the definition of $g^n$ and $g$, we have
$$
\begin{aligned}
&|g(s,y,E[Y],\mathbb{P}_{(\mathring{Y},V_s)})-g^n(s,y,\mathbb{P}_{(Y,V_s)})|\\
\leq&\int_{\mathbb{R}^2}n^2\big|g(s,y,E[Y],\mathbb{P}_{(\mathring{Y},V_s)})-g(s,y-y',E[Y]-y'',\mathbb{P}_{( \mathring{Y},V_s)})\big|\rho(ny',ny'')dy'dy''\\
=&\int_{\mathbb{R}^2}\big|g(s,y,E[Y],\mathbb{P}_{(\mathring{Y},V_s)})-g(s,y-\frac{y'}{n},E[Y]-\frac{y''}{n},\mathbb{P}_{( \mathring{Y},V_s)})\big|\rho(y',y'')dy'dy'',
\end{aligned}
$$
which  converges to $0$ by the continuity of $g$ and the dominated convergence theorem, as $n\rightarrow\infty$.
The proof is complete.
\end{proof}
\begin{remark}
For any given $V,\ V'\in L^2_{\mathbb{F}}(0,T;\mathbb{R})$, 
we also have the following property:
$$
|g^n(s,y,\mathbb{P}_{(Y,V_s)})-g^n(s,y,\mathbb{P}_{(Y,V'_s)})|\leq CW_2(\mathbb{P}_{(Y,V_s)},\mathbb{P}_{(Y,V'_s)}),
$$
for all $s\in[0,T],\ y\in\mathbb{R},\ Y\in L^2(\mathcal{F})$.
\end{remark}

\begin{proof}[\textbf{Proof of Proposition \ref{111}}] The uniqueness is proved similar to the uniqueness proof of Theorem \ref{le2.4}. Let us prove the existence. For this, we  write $$g(s,y,E[Y],\mathbb{P}_{(\mathring{Y},V_s)}):=\widetilde{f}(s,y,V_s,E[Y],\mathbb{P}_{(\mathring{Y},V_s)})=f(s,y,V_s,\mathbb{P}_{(Y,V_s)}),$$ for all $y\in \mathbb{R},~Y\in L^2(\mathcal{F})$, where $\mathring{Y}:=Y-E[Y]$. We first approximate $g$ by  $g^n$, where $g^n$ is defined in \eqref{juanji}. Our assumptions (A2) and (A3) imply that
\begin{equation}\label{gndandiao}
E[(g^n(s,Y_1,\mathbb{P}_{(Y_1,V_s)})-g^n(s,Y_2,\mathbb{P}_{(Y_2,V_s)}))(Y_1-Y_2)]\leq CE[|Y_1-Y_2|^2],
\end{equation}
where the constant $C$ does not depend on $n$. Then, thanks to Lemma \ref{lemmajuanji} $(\rm i)$,   $(\rm ii)$, and by using an  argument similar to that of Theorem 2.1 in \cite{JMAA}, we get that the mean-field BSDE
\begin{equation}\label{mfbsden}
Y^n_t=\xi+\int_t^T g^n\left(s, Y^n_s, \mathbb{P}_{(Y^n_s,V_s)}\right) d s-\int_t^T Z^n_s d W_s
\end{equation}
has a unique solution $(Y^n,\ Z^n)\in \mathcal{S}^2(0,T;\mathbb{R})\times L^2_{\mathbb{F}}(0,T;\mathbb{R})$. Moreover,
$$
\left|Y^n_t\right|^2+\int_t^T|Z^n_s|^2 d s =|\xi|^2+2 \int_t^T Y^n_s \cdot g^n\left(s, Y^n_s, \mathbb{P}_{(Y^n_s,V_s)}\right) d s-2 \int_t^T Y^n_s \cdot Z^n_s d W_s,$$
and
\begin{equation*}
E\left[\left|Y^n_t\right|^2+\int_t^T|Z^n_s|^2 d s\right] \leq E\left[|\xi|^2\right]+C E\left[\int_t^T\left(1+\left|Y^n_s\right|^2\right) d s\right].
\end{equation*}
From these both relations above and Lemma \ref{lemmajuanji} (ii) it follows by standard estimates that
\begin{equation}\label{BSDEgn}
\sup _n E\Big[\sup_{t\in[0,T]} |Y^n_t|^2+\int_0^T|Z^n_s|^2 d s\Big]<\infty.
\end{equation}
For $n,~m\geq1$, we put $\bar{Y}^{n, m}:=Y^n-Y^m$, $\bar{Z}^{n, m}:=Z^n-Z^m$;  $\bar{g}^{n, m}:=g^n-g^m$, $\bar{g}^n:=g^n-g$.  From \eqref{mfbsden}, by applying It\^{o} formula, we get
$$
\begin{aligned}
& E\left[|\bar{Y}_t^{n, m}|^2+\int_t^T|\bar{Z}_s^{n, m}|^2 d s\right] \\
& =2 E\left[\int_t^T\bar{Y}_s^{n, m}(g^n(s, Y^n_s, \mathbb{P}_{(Y^n_s,V_s)})-g^m(s, Y^m_s, \mathbb{P}_{(Y^m_s,V_s)})) d s\right] \\
& =2 E\left[\int_t^T\bar{Y}_s^{n, m}(g^n(s, Y^n_s, \mathbb{P}_{(Y^n_s,V_s)})-g^n(s, Y^m_s, \mathbb{P}_{(Y^m_s,V_s)})) d s\right] \\
 &~~~~+2 E\left[\int_t^{T}\bar{Y}_s^{n, m} \bar{g}^{n, m}(s, Y_s^m, \mathbb{P}_{(Y^m_s,V_s)})) d s\right] \\
& =I_t^{n, m}+J_t^{n, m},~ t \in[0, T],
\end{aligned}
$$
where $$I_t^{n, m}:=2 E\left[\int_t^T\bar{Y}_s^{n, m}(g^n(s, Y^n_s, \mathbb{P}_{(Y^n_s,V_s)})-g^n(s, Y^m_s, \mathbb{P}_{(Y^m_s,V_s)})) d s\right],~t \in[0, T]$$and$$ J_t^{n, m}:=2 E\left[\int_t^{T}\bar{Y}_s^{n, m} \bar{g}^{n, m}(s, Y_s^m, \mathbb{P}_{(Y^m_s,V_s)}) d s\right],~t \in[0, T].$$
Thanks to \eqref{gndandiao}, $I_t^{n, m} \leq 2 E\left[\displaystyle\int_t^T\left|\bar{Y}_s^{n, m}\right|^2 d s\right], ~t \in[0, T]$.
Moreover, by $\bar{g}^{n, m}=\bar{g}^n-\bar{g}^m$, we obtain
$$
\begin{aligned}
J_t^{n, m}= & 2 E\left[\int_t^T\bar{Y}_s^{n, m} \bar{g}^n(s, Y_s^m, \mathbb{P}_{(Y^m_s,V_s)}) d s\right]  -2 E\left[\int_t^T\bar{Y}_s^{n, m} \bar{g}^m(s, Y_s^m, \mathbb{P}_{(Y^m_s,V_s)}) d s\right].
\end{aligned}
$$
Observe that, for all $\ell \geq 1$,
$$
\begin{aligned}
& \left|2E\left[\int_t^T\bar{Y}_s^{n, m} \bar{g}^{\ell}(s, Y_s^m, \mathbb{P}_{(Y^m_s,V_s)}) d s\right]\right| \\
& \leq E\left[\int_t^T| \bar{Y}_s^{n, m}|^2 ds\right]+E\left[\int_t^T|\bar{g}^{\ell}(s, Y_s^m, \mathbb{P}_{(Y^m_s,V_s)})|^2 d s\right],
\end{aligned}
$$
where
$$
\begin{aligned}
& E\left[\int_t^T|\bar{g}^{\ell}(s, Y_s^m, \mathbb{P}_{(Y^m_s,V_s)})|^2 d s\right]\\
& \leq E\left[\int_t^T\int_{\mathbb{R}^2}\ell^2\rho(\ell y',\ell y'')|g(s,Y_s^m-y',E[Y_s^m]-y'',\mathbb{P}_{( \mathring{Y_s^m},V_s)})\right.\\
&\left.~~~~~~~~~~~~~~~~~~~~~~~~~~~~~~~~~~~-g(s,Y_s^m,E[Y_s^m],\mathbb{P}_{( \mathring{Y_s^m},V_s)}) |^2 dy'dy''ds\right].
\end{aligned}
$$
Note that, from the assumptions (A1), (A2) and (A4),
\begin{equation}\label{linearg}
|g(s, \omega,y,m,\mathbb{P}_{(\mathring{\xi},V_s)})|\leq C_{K,\alpha_1,\alpha_4}(1+|y|+|m|+|V_s|+(E[|\mathring{\xi}|^2])^{\frac{1}{2}}+(E[|V_s|^2])^{\frac{1}{2}}),
\end{equation}
and for $M,~C_{\ast}>0$, consider the continuity modulus
$$
\begin{aligned}
& \mathcal{X}_{M, C_{\ast}}(s, \omega, \delta):=\sup \left\{\left|g\left(s, \omega, y-y', m-y'', \mathbb{P}_{(\mathring{\xi},V_s)}\right)-g\left(s, \omega, y, m, \mathbb{P}_{(\mathring{\xi},V_s)}\right)\right|,\right.\\
&\left.~~~~~~~~~~~~~~~~~~~~~~~~~~~~~~|y'|,~|y''| \leq \delta,~|y| \leq M,~
 |m| \leq C_1,~(E[|\mathring{\xi}|^2])^{\frac{1}{2}} \leq C_{\ast}\right\}.
\end{aligned}
$$
Obviously, from the assumption \eqref{A15}, $\mathcal{X}_{M, C_{\ast}}(s, \omega, \delta) \rightarrow 0$, as $\delta \downarrow 0$, $ds\mathbb{P}(d\omega)$-a.e. Here, $\displaystyle C_{\ast}:=\sup_{n\geq1}\Big(E\Big[\sup_{t\in[0,T]} |Y^n_t|^2\Big]\Big)^{\frac{1}{2}}<+\infty$, because of \eqref{BSDEgn}.

\noindent Let $M$ be sufficiently large. Then,
$$
\begin{aligned}
& E\left[\int_0^T|\bar{g}^{\ell}(s, Y_s^m, \mathbb{P}_{(Y^m_s,V_s)})|^2 d s\right]\\
&\leq E\left[\int_0^T\int_{\mathbb{R}^2}\ell^2\rho(\ell y',\ell y'') \mathcal{X}^2_{M, C_{\ast}}\Big(s, \frac{1}{\ell}\Big) dy'dy''ds\mathbf{1}_{\{\sup_{s\in[0,T]} |Y^m_s|\leq M\}}\right]\\
&~~~~+E\left[\int_0^T\int_{\mathbb{R}^2}\ell^2\rho(\ell y',\ell y'') C_{K,\alpha_1,\alpha_4,C_{\ast}}(1+M+|Y^m_s|+|V_s|+(E[|V_s|^2])^{\frac{1}{2}}) dy'dy''ds\right.\\
&~~~~~~~~~~~~~\left.\cdot \mathbf{1}_{\{\sup_{s\in[0,T]} |Y^m_s|> M\}}\right].
\end{aligned}
$$
As, from H\"{o}lder inequality and Chebyshev's inequality,
$$
\begin{aligned}
& E[(|Y^m_s|+|V_s|)\mathbf{1}_{\{\sup_{s\in[0,T]} |Y^m_s|> M\}}]\\
&\leq C(E[(|Y^m_s|^2+|V_s|^2)])^\frac{1}{2}\Big(\mathbb{P}\Big\{\sup_{s\in[0,T]} |Y^m_s|> M\Big\}\Big)^\frac{1}{2}\\
&\leq\frac{C}{M}E\Big[\sup_{s\in[0,T]} |Y^m_s|^2\Big]+C(E[|V_s|^2])^\frac{1}{2}\Big(E\Big[\sup_{s\in[0,T]} \frac{|Y^m_s|^2}{M^2}\Big]\Big)^\frac{1}{2}\\
&\leq C\Big(\frac{C^2_{\ast}}{M}+\frac{C_{\ast}}{M}(E[|V_s|^2])^\frac{1}{2}\Big),
\end{aligned}
$$
then by using Jensen's inequality, we obtain
$$
\begin{aligned}
& E\left[\int_0^T|\bar{g}^{\ell}(s, Y_s^m, \mathbb{P}_{(Y^m_s,V_s)})|^2 d s\right]\\
&\leq E\left[\int_0^T\mathcal{X}^2_{M, C_{\ast}}\Big(s, \frac{1}{\ell}\Big)ds\mathbf{1}_{\{\sup_{s\in[0,T]} |Y^m_s|\leq M\}}\right]\\
&~~~~+C_{T,K,\alpha_1,\alpha_4,C_{\ast}}\Big(\frac{1+M}{M^2}+\frac{1}{M}+\frac{1}{M}\Big(\int_0^T E[|V_s|^2]ds\Big)^\frac{1}{2}\Big).
\end{aligned}
$$
Observe that, thanks to \eqref{linearg},
$$
\mathcal{X}_{M, C_{\ast}}\Big(s, \frac{1}{\ell}\Big) \leq C_{K,\alpha_1,\alpha_4,C_{\ast}}(1+M+|V_s|+(E[|V_s|^2])^{\frac{1}{2}}).
$$
By combining the above estimates, we get
$$
\begin{aligned}
& E\left[|\bar{Y}_t^{n, m}|^2+\int_t^T|\bar{Z}_s^{n, m}|^2 d s\right] \\
& \leq C E\left[\int_t^{T} |\bar{Y}_s^{n, m}|^2 d s\right]+\sup_{\ell \geq m \wedge n} E\left[\int_0^T \mathcal{X}_{M, C_{\ast}}^2\Big(s, \frac{1}{\ell}\Big) d s\right]\\
&~~~+C_{T,K,\alpha_1,\alpha_4,C_{\ast}}\Big(\frac{1}{M^2}+\frac{1}{M}\Big(1+\Big(E\Big[\int_0^T|V_s|^2 d s\Big]\Big)^{\frac{1}{2}}\Big)\Big), \quad t \in[0, T],
\end{aligned}
$$
and by Gronwall's inequality,
$$
\begin{aligned}
&\sup_{t\in[0,T]}E\left[|\bar{Y}_t^{n, m}|^2+\int_0^T|\bar{Z}_s^{n, m}|^2 d s\right]\\
& \leq C \sup_{\ell \geq m \wedge n} E\left[\int_0^T \mathcal{X}_{M,C_{\ast}}^2\Big(s, \frac{1}{\ell}\Big) d s\right]+C_{T,K,\alpha_1,\alpha_4,C_{\ast}}\Big(\frac{1}{M^2}+\frac{1}{M}\Big(1+\Big(E\Big[\int_0^T|V_s|^2 d s\Big]\Big)^{\frac{1}{2}}\Big)\Big).
\end{aligned}
$$
Hence, from the dominated convergence theorem,
$$
\begin{aligned}
& \varlimsup_{m,~ n \rightarrow \infty} \sup_{t \in[0, T]} E\left[|\bar{Y}_t^{n, m}|^2+\int_0^T|\bar{Z}_s^{n, m}|^2 d s\right]\\
& \leq C_{T,K,\alpha_1,\alpha_4,C_{\ast}}\Big(\frac{1}{M^2}+\frac{1}{M}\Big(1+\Big(E\Big[\int_0^T|V_s|^2 d s\Big]\Big)^{\frac{1}{2}}\Big)\Big),
\end{aligned}
$$
then taking the limit as $M \rightarrow \infty$, we get
$$\lim _{m, n \rightarrow \infty} \sup _{t\in[ 0, T]} E\left[|\bar{Y}_t^{n, m}|^2+\int_0^T|\bar{Z}_s^{n, m}|^2 d s\right]=0,$$
this implies  $Y^n \rightarrow Y$ in $L_{\mathbb{F}}^2(0, T ; \mathbb{R})$ and $Z^n \rightarrow Z$ in $L_{\mathbb{F}}^2(0, T ; \mathbb{R})$.
Hence, similar to the proof of (iii) in Lemma \ref{lemmajuanji}, we get
$$
\begin{aligned}
E\Big[\sup _{s \in[0, T]}\left|\int_s^T g_n(r, Y_r^n,\mathbb{P}_{(Y_r^n,V_r)}) d r-\int_s^T f(r, Y_r, V_r, \mathbb{P}_{(Y_r,V_r)}) d r\right|^2\Big]
 \rightarrow 0,~ n \rightarrow \infty .
\end{aligned}
$$
From the Burkholder-Davis-Gundy inequality we can get
$$
\begin{aligned}
&E\Big[\sup _{s \in[0, T]}\left|\int_s^T Z_r^n d W_r-\int_s^T Z_r d W_r\right|^2\Big] \leq C E\Big[\int_0^T|Z_r^n-Z_r|^2 d r\Big] \rightarrow 0,~ n \rightarrow \infty .
\end{aligned}
$$
Standard arguments allow now to conclude that $(Y,Z)\in \mathcal{S}^2(0,T;\mathbb{R})\times L_\mathbb{F}^2(0,T;\mathbb{R})$ solves \eqref{MFBSDEV}.
\end{proof}
\begin{example}\label{Ex2.5}\rm
Consider the following mean-field BSDE:
\begin{equation}\label{EX1}
Y_t=\xi+\int_t^T\big(-\sqrt{Y_s^+\wedge 1}-\sqrt{(E[Y_s])^+\wedge 1}+f_1(Z_s,E[h(Y_s,Z_s)])\big)ds-\int_t^TZ_s dW_s,~t\in[0,T].
\end{equation}
The driving coefficient $f(s,y,z,\mathbb{P}_{(\xi,\eta)})=-\sqrt{y^+\wedge 1}-\sqrt{(E[\xi])^+\wedge 1}+f_1(z,E[h(\xi,\eta)])$ is not Lipschitz, although we suppose $f_1:\mathbb{R}^2 \rightarrow \mathbb{R}$, $h:\mathbb{R}^2 \rightarrow \mathbb{R}$ bounded and Lipschitz, but
$$
\widetilde{f}(s,y,z,m,\mathbb{P}_{(\mathring{\xi},\eta)}) = -\sqrt{y^+\wedge 1}-\sqrt{m^+\wedge 1} + f_1(z,E[h(m+\mathring{\xi},\eta)]),
$$
where $(s,y,z)\in [0,T]\times \mathbb{R}\times\mathbb{R}$, $\xi,\ \eta\in L^2(\mathcal{F};\mathbb{R})$, satisfies our assumptions of Theorem \ref{le2.4}.
\end{example}
\begin{example}\label{EX2}\rm
Unlike the non-Lipschitz Example \ref{Ex2.5}, let us still consider an example with unbounded derivative which satisfies our assumptions too. We denote that
$$
t_0:=0,\quad t_{2\ell+1}:=t_{2\ell}+\frac{\pi}{\ell+1},\ \text{for }\ell \geq 0; \quad \ t_{2\ell}:=t_{2\ell -1}+\pi,\ \text{for } \ell \geq 1.
$$
The function $g$ is defined as
$$
g(y):= \mathbf{1}_{(-\infty,0)}(y)+\sum_{\ell\geq 0}\big\{\mathbf{1}_{[t_{2\ell},t_{2\ell+1})}(y)\cos\big((\ell +1)(y-t_{2\ell})\big)+\mathbf{1}_{[t_{2\ell +1},t_{2\ell+2})}(y)\cos(y-t_{2\ell+1}+\pi)\big\},
$$
for $y\in\mathbb{R}$. Then $g\in C^1(\mathbb{R})$ bounded by $1$, and the derivative of $g$ is
$$
g'(y)=\sum_{\ell\geq 0}\big\{-\mathbf{1}_{[t_{2\ell},t_{2\ell+1})}(y)(\ell +1)\sin\big((\ell+1)(y-t_{2\ell})\big)
-\mathbf{1}_{[t_{2\ell +1},t_{2\ell+2})}(y)\sin(y-t_{2\ell+1}+\pi)\big\},\quad y\in \mathbb{R}.
$$
This implies that  $0\leq g'(y)\leq 1$ on $[t_{2\ell+1},t_{2\ell+2})$;
and  $0\geq g'(y)\geq -\ell-1$ on $[t_{2\ell},t_{2\ell+1})$.
But $ t_{2\ell} = t_{2\ell-1}+\pi > t_{2\ell-2}+\pi > t_{2\ell-4}+2\pi > \cdots >t_0+\ell\pi = \ell\pi$, $\displaystyle\ell < \frac{1}{\pi}t_{2\ell}$, and so
$\displaystyle0 \geq g'(y) \geq -\frac{1}{\pi}t_{2\ell}-1 \geq -\frac{y}{\pi}-1$, for $y\in[t_{2\ell},t_{2\ell+1})$.
We also remark that for any $\ell\geq 0$, there exists $y_\ell \in (t_{2\ell}, t_{2\ell+1})$ such that $g'(y_{\ell})= -(\ell +1)(\rightarrow -\infty,\ \text{as }\ell \uparrow +\infty)$. Consequently,
\begin{itemize}
\item[$\rm(i)$] $g\in C^1(\mathbb{R})$ is bounded;
\item[$\rm(ii)$] $\displaystyle -(\frac{y}{\pi}+1)\leq g'(y)\leq 1$, for any $y\in\mathbb{R}$;
\item[$\rm(iii)$] $\displaystyle \liminf\limits_{y\rightarrow +\infty}\frac{g'(y)}{y}= -\frac{1}{\pi}.$
\end{itemize}
In particular, it follows $(y_2-y_1)(g(y_2) - g(y_1))\leq (y_2-y_1)^2$, for any $y_1,\ y_2\in\mathbb{R}$. Indeed, for $y_2>y_1$, and $\theta_{y_1,y_2}\in (y_1,y_2)$ such that $g(y_2) - g(y_1) = (y_2-y_1)g'(\theta_{y_1,y_2})$, note that
\begin{equation*}
g'(\theta_{y_1,y_2})
\left\{
\begin{aligned}
&= 0,\quad \theta_{y_1,y_2}<0;\\
&\leq 0,\quad \theta_{y_1,y_2}\in [t_{2\ell},t_{2\ell+1}),\ \ell\geq 0;\\
&\leq 1,\quad \theta_{y_1,y_2}\in [t_{2\ell+1},t_{2\ell+2}),\ \ell\geq 0,\\
\end{aligned}
\right.
\end{equation*}
and so
$$
(y_2-y_1)(g(y_2)-g(y_1)) = (y_2-y_1)^2g'(\theta_{y_1,y_2})\leq (y_2-y_1)^2.
$$
Consider the following BSDE:
\begin{equation}\label{Ex0}
Y_t=\xi+\int_t^T \big(g(Y_s)+g(E[Y_s])+f_1(Z_s,E[h(Y_s,Z_s)])\big)ds-\int_t^TZ_s dW_s,~t\in[0,T],
\end{equation}
where $\xi\in L^2(\mathcal{F}_T;\mathbb{R})$, $f_1:\mathbb{R}^2 \rightarrow \mathbb{R}$, $h:\mathbb{R}^2 \rightarrow \mathbb{R}$ are bounded and Lipschitz. From Theorem \ref{le2.4}, BSDE \eqref{Ex0} has a unique solution $(Y,Z)\in\mathcal{S}^2(0,T;\mathbb{R})\times L^2_{\mathbb{F}}(0,T;\mathbb{R})$.
\end{example}

\begin{example}\rm
Consider the following mean-field BSDE:
\begin{equation}\label{EX3}
Y_t=\xi+\int_t^T(l(Y_s)+l(E[Y_s])+Z_s+E[Z_s])ds-\int_t^TZ_s dW_s,~t\in[0,T],
\end{equation}
where $\xi\in L^2(\mathcal{F}_T;\mathbb{R})$ and
\begin{equation}\label{exl}
l(y)=\left\{
             \begin{array}{l}
           0,~~~~~~~~~~y\in(-\infty,0],\\
-\sqrt{y} ,~~~~~y\in(0,1],\\
-e^{1-y} ,~~~y\in(1,\infty].
             \end{array}
\right.
\end{equation} Let $\displaystyle f(y,z,\mathbb{P}_{(\xi,\eta)}):=l(y)+l(E[\xi])+z+E[\eta]$, for $y,~z\in\mathbb{R}$, $\xi,~\eta\in L^2(\mathcal{F};\mathbb{R})$. Then, we have $\displaystyle \widetilde{f}(y,z,m,\mathbb{P}_{(\mathring{\xi},\eta)}):=l(y)+l(E[\mathring{\xi}+m])+z+E[\eta]=l(y)+l(m)+z+E[\eta]$, where $m\in\mathbb{R}$ and $\mathring{\xi}\in L_0^2(\mathcal{F};\mathbb{R})$. Obviously, the function $f$ does not satisfy the Lipschitz continuity, but
  satisfies (A1). And $\widetilde{f}$ satisfies (A2) and (A3)(\ref{A15}). Moreover, for all $y,~y'\in\mathbb{R}$, we have
$$
\begin{aligned}
&(y-y')(\widetilde{f}(y,z,m,\mathbb{P}_{(\mathring{\xi},\eta)})-\widetilde{f}(y',z,m,\mathbb{P}_{(\mathring{\xi},\eta)}))=(l(y)-l(y'))(y-y')
\leq C|y-y'|^2,
\end{aligned}
$$
where $C>0$ is a constant, which implies (2.5). Similarly, for all $\xi,~\xi'\in L^2(\mathcal{F};\mathbb{R})$,
$$
\begin{aligned}
E[(\xi-\xi')(\widetilde{f}(y,z,E[\xi],\mathbb{P}_{(\mathring{\xi},\eta)})-\widetilde{f}(y,z,E[\xi'],\mathbb{P}_{(\mathring{\xi},\eta)}))]\leq CE[|\xi-\xi'|^2],
\end{aligned}
$$
which implies (2.6). Consequently, from Theorem \ref{le2.4}, BSDE \eqref{EX3} has a unique solution $(Y,Z)\in\mathcal{S}^2(0,T;\mathbb{R})\times L^2_{\mathbb{F}}(0,T;\mathbb{R})$.
\end{example}
\section{Mean-field SDEs with monotonicity condition}
Let us formulate the following assumptions:
\begin{itemize}
 \item[$\rm(B1)$]  $b=b(s,x,\mu):[0,T]\times\Omega\times\mathbb{R}\times\mathscr{P}_2(\mathbb{R})\mapsto{\mathbb{R}}$ is $\mathbb{F}$-adapted, for all $(x,\mu)\in \mathbb{R}\times\mathscr{P}_2(\mathbb{R})$, and there exists a constant\ $K\ge 0$,\ such that\ $\mathbb{P}$-a.s.,\ for all\ $t\in[0,T],\ x\in\mathbb{R},\ X\in L^2(\mathcal{F};\mathbb{R})$,\
$$|b(t,x,\mathbb{P}_{X})|\leq K\big(1+|x|+\big(E[|X|^2]\big)^{\frac{1}{2}}\big).$$

 \item[$\rm(B2)$]
With the notation
$$
\widetilde{b}(s,\omega,x,m,\mathbb{P}_{\mathring{\xi}}):= b(s,\omega,x,\mathbb{P}_{\mathring{\xi}+m}),
$$
$\widetilde{b}$ satisfies the following conditions: There exists a constant\ $\alpha_1>0$, such that
\begin{equation}\label{A12a}
\begin{split}
&\big|\widetilde{b}(s,\omega,x,m,\mathbb{P}_{\mathring{\xi}})-\widetilde{b}(s,\omega,x,m,\mathbb{P}_{\mathring{\xi}'})\big|\le \alpha_1 W_2(\mathbb{P}_{\mathring{\xi}},\mathbb{P}_{\mathring{\xi}'}),
\end{split}
\end{equation}
 for all\ $s\in[0,T],\ \omega\in\Omega,\ x,\ m\in\mathbb{R},\ \mathring{\xi},\ \mathring{\xi}'\in L_0^2(\mathcal{F};\mathbb{R})\big(\text{defined in assumption (A2)}\big)$.

\item[$\rm(B3)$] The coefficient $\widetilde{b}$ is continuous w.r.t. $x$ and $m$, and there exist constants\ $C,\ \alpha_2,\ \alpha_3\ge 0$,\ such that for all\ $s\in[0,T],\ \omega \in \Omega,\ x,\ x',\ m,\ m'\in\mathbb{R},\ \mathring{\xi}\in L_0^2(\mathcal{F};\mathbb{R}),\ \theta,\ X,\ X' \in L^2(\mathcal{F};\mathbb{R})$,
\begin{equation}\label{A13a}
\begin{split}
(\widetilde{b}(s,\omega,x,m,\mathbb{P}_{\mathring{\xi}})-\widetilde{b}(s,\omega,x',m,\mathbb{P}_{\mathring{\xi}}))(x-x') \le \alpha_2{\vert x-x' \vert}^2,
\end{split}
\end{equation}
\begin{equation}\label{A14a}
\begin{split}
E\big[(\widetilde{b}(s,\theta,E[X],\mathbb{P}_{\mathring{\xi}})-\widetilde{b}(s,\theta,E[X'],\mathbb{P}_{\mathring{\xi}}))(X-X')\big] \le \alpha_3E\big[{\vert X-X' \vert}^2\big].
\end{split}
\end{equation}
\begin{equation}\label{2.7'}
\begin{aligned}[c]
&\text{Moreover, }(x,m)\rightarrow \widetilde{b}(s,\omega,x,m,\mathbb{P}_{\mathring{\xi}})\text{ is uniformly continuous on compacts,}\\ &\text{uniformly w.r.t. }\mathbb{P_{\mathring{\xi}}},\ ds\mathbb{P}(d\omega)\text{-a.e.}
\end{aligned}
\end{equation}

\item[$\rm(B4)$]$\sigma=\sigma(t,x,\mu):[0,T]\times\Omega\times{\mathbb{R}}\times\mathscr{P}_2(\mathbb{R})\mapsto{\mathbb{R}}$
  is Lipschitz w.r.t.\ $x,\ \mu$, i.e., there exists a constant\ $C_1>0$,\  such that $\mathbb{P}$-a.s., for all\ $t\in [0,T],\ x_1,\ x_2\in \mathbb{R},\ \mu_1,\ \mu_2\in\mathscr{P}_2(\mathbb{R})$,\
 $$\vert \sigma(t,x_1,\mu_1)-\sigma(t,x_2,\mu_2)\vert\le C_1(|x_1-x_2|+W_2(\mu_1,\mu_2)).$$

 \item[$\rm(B5)$]
 $\sigma(\cdot,0,\delta_0)\in L_{\mathbb{F}}^p(0,T;\mathbb{R}),\ p\ge2$.
 \end{itemize}
\begin{theorem}\label{le2.5}
Under the assumptions $\rm(B1)$-$\rm(B5)$, the following mean-field stochastic differential equation
 \begin{equation}\label{eq 2.13}
\left\{
             \begin{array}{l}
           dX_t=b(t,X_t,\mathbb{P}_{X_t})dt+\sigma(t,X_t,\mathbb{P}_{X_t})dW_t,~t\in[0,T],\\
X_0=x_{0} ,
             \end{array}
\right.
  \end{equation}
 has a unique solution\ $X\in S^2(0,T;\mathbb{R})$.

 Moreover, for all $p\geq2$, there exists a constant $C=C (T, K,p)>0$ only depending $K$, $p$ and $T$, such that the following estimate holds true,
 \begin{equation}\label{eq 2.14}
\begin{aligned}
E[\sup_{s\in[0,T]}|X_{s}|^p]\le C(p,K,T)(1+|x_0|^p).
\end{aligned}
\end{equation}
\end{theorem}

 \begin{proof}[Proof.]
\indent We show the uniqueness first.
Suppose\ $X^1,~X^2\in S^2(0,T;\mathbb{R})$ are solutions to equation (\ref{eq 2.13}), and put\ $\widehat X=X^1-X^2$. Applying  It\^{o}'s formula to $|\widehat X|^2$ and taking expectation, with (B2), (B3) and (B4), we have
\begin{align*}
  &E[|\widehat X_t|^2]\\
  =&E\left[\int_0^t \left\{2\widehat X_s(b(s,X^1_s,\mathbb{P}_{X^1_s})-b(s,X^2_s,\mathbb{P}_{X^2_s}))+|\sigma(s,X^1_s,\mathbb{P}_{X^1_s})-\sigma(s,X^2_s,\mathbb{P}_{X^2_s})|^2\right\} ds\right]\\
\le&2C^2_1 E\left[\int_0^t(|\widehat X_s|^2+E[|\widehat X_s|^2])ds\right]
+2 E\left[\int_0^t \widehat X_s(\widetilde{b}(s,X^1_s,E[X^1_s],\mathbb{P}_{\mathring{X}^1_s})-\widetilde{b}(s,X^2_s,E[X^1_s],\mathbb{P}_{\mathring{X}^1_s}))ds\right]\\
&+2E\left[\int_0^t \widehat X_s(\widetilde{b}(s,X^2_s,E[X^1_s],\mathbb{P}_{\mathring{X}^1_s})-\widetilde{b}(s,X^2_s,E[X^2_s],\mathbb{P}_{\mathring{X}^1_s})) ds\right]\\
&+2E\left[\int_0^t \widehat X_s(\widetilde{b}(s,X^2_s,E[X^2_s],\mathbb{P}_{\mathring{X}^1_s})-\widetilde{b}(s,X^2_s,E[X^2_s],\mathbb{P}_{\mathring{X}^2_s})) ds\right]\\
\le&(4C^2_1+\alpha_1^2 +2\alpha_2+2\alpha_3+2)E\left[\int_0^t|\widehat X_s|^2ds\right],
 \end{align*}
where $\mathring{X}:=X-E[X]$. Thus, from Gronwall's inequality, we have  $E[|\widehat X_t |^2]=0$, i.e., $X^1_t=X^2_t,~t\in[0,T]$, $\mathbb{P}$-a.s.\\
\indent
Using an argument similar to that of proof for Theorem 2.1, we  prove the existence of the solution $X\in S^2(0,T;\mathbb{R})$.
 Let us now prove estimate (\ref{eq 2.14}). Let $n\geq 1$ and $\tau_n:= \inf\{t\geq 0: |X_t|\geq n\}\wedge T$. By applying  It\^{o}'s formula to  $(|X_t|^2)^\frac{p}{2}$  and taking expectation we have,  for $t\in[0,T]$,
\begin{align*}
 &E[\sup_{s\in[0,t]}|X_{s\wedge\tau_n}|^p]\\ \leq&|x_0|^p+pE\left[\int_0^{{t\wedge\tau_n}}|X_s|^{p-2}(X_sb(s,X_s,\mathbb{P}_{X_s}))^+ds\right]
 +\frac{p(p-1)}{2}E\left[\int_0^{{t\wedge\tau_n}}|X_s|^{p-2}|\sigma(s,X_s,\mathbb{P}_{X_s})|^2ds\right]\\
 &+pE\Big[\sup_{s\in[0,t]}\Big|\int_0^{{s\wedge\tau_n}}|X_r|^{p-2}X_r\sigma(r,X_r,\mathbb{P}_{X_r})dW_r\Big|\Big]\\
 \le&|x_0|^p+CKpE\left[\int_0^{{t\wedge\tau_n}}|X_s|^{p-1}ds\right]+\frac{1}{4}E[\sup_{s\in[0,t]}|X_{s\wedge\tau_n}|^p]\\
 &+C_p E\left[\int_0^{{t\wedge\tau_n}}|X_s|^{p-2}|\sigma(s,X_s,\mathbb{P}_{X_s})-\sigma(s,0,\delta_0)+\sigma(s,0,\delta_0)|^2ds\right],~t\in[0,T].
 \end{align*}
 Thus,
 \begin{align*}
 &E[\sup_{s\in[0,t]}|X_{s\wedge\tau_n}|^p]\\
 \le&|x_0|^p+C_{p,K}E\Big[\int_0^{{t\wedge\tau_n}}\Big(|X_s|^{p-1}+|X_s|^{p-2}|\sigma(s,0,\delta_0)|^2+|X_s|^{p}
+|X_s|^{p-2}E[|X_s|^2]\Big)ds\Big],~t\in[0,T].
\end{align*}
By using Young's inequality, (B5) and $X\in S^2(0,T;\mathbb{R})$, we obtain
\begin{equation*}
\begin{aligned}
 E[\sup_{s\in[0,t]}|X_{s\wedge\tau_n}|^p]&\le |x_0|^p+C_{p,K}E\left[\int_0^{t\wedge\tau_n}\left(1+|X_s|^{p}+(E[|X_s|^2])^{\frac{p}{2}}\right)ds\right]\\
 &\leq|x_0|^p+C_{p,K}E\left[\int_0^{t}\left(1+|X_{s\wedge\tau_n}|^{p}+\big(E[|X_s|^2]\big)^\frac{p}{2}\right)ds\right],~t\in[0,T].
\end{aligned}
\end{equation*}
Hence, for all $t\in [0,T]$,
\begin{equation*}
\begin{aligned}
E[\sup_{s\in[0,t]}|X_{s\wedge\tau_n}|^p]\le |x_0|^p+C_{p,K}\int_0^{t}\Big(1+E[\sup_{r\in[0,s]}|X_{r\wedge\tau_n}|^{p}]+\big(E[|X_s|^2]\big)^\frac{p}{2}\Big)ds,
\end{aligned}
\end{equation*}
and by Gronwall's Lemma we obtain $$E[\sup_{s\in[0,t]}|X_{s\wedge\tau_n}|^p]\le C(p,K,T)\Big(1+|x_0|^p+\int_0^t\big(E[|X_s|^2]\big)^\frac{p}{2}ds\Big),\quad t\in [0,T],$$
for $n\geq1$, where $C(p,K,T)$ is a constant depending on $p$, $K$ and $T$. Since none of the constants
used in the above computation depends on $n$, letting
$n\rightarrow\infty$ by using the monotone convergence theorem yields
$$
E[\sup_{s\in[0,t]}|X_s|^p]\leq C(p,K,T)\Big(1+|x_0|^p+\int_0^t\big(E[|X_s|^2]\big)^\frac{p}{2}ds\Big),\quad t\in [0,T].
$$
Since $X\in S^2(0,T;\mathbb{R})$, this proves that $E[\sup_{s\in [0,T]}|X_s|^p]<+\infty$, and using H\"{o}lder's inequality,
$\big(E[|X_s|^2]\big)^{\frac{p}{2}}\leq E[|X_s|^p]$, followed by Gronwall's lemma we can conclude.
\end{proof}
\begin{example}\rm
In analogy to (\ref{EX1}), we obtain that the following mean-field stochastic differential equation
\begin{equation}
\left\{
\begin{aligned}
&dX_t =\big(-\sqrt{X_t^+\wedge 1}-\sqrt{(E[X_t])^+\wedge 1}+b_1(X_t,E[h(X_t)])\big)dt + \sigma_1 (X_t,E[l(X_t)])dW_t,\ t\in [0,T],\\
&X_0 = x_0\in\mathbb{R},
\end{aligned}\right.
\end{equation}
has a unique solution $X\in S^2(0,T;\mathbb{R})$, for $b_1,\ \sigma_1:\mathbb{R}^2\rightarrow \mathbb{R}$, $h,\ l:\mathbb{R}\rightarrow\mathbb{R}$ bounded and Lipschitz.
\end{example}
\begin{example}\rm
The following mean-field stochastic differential equation
\begin{equation}
\left\{
\begin{aligned}
&dX_t =\big(g(X_t)+g(E[X_t])+b_1(X_t,E[h(X_t)])\big)dt + \sigma_1 (X_t,E[l(X_t)])dW_t,\ t\in [0,T],\\
&X_0 = x_0\in\mathbb{R},
\end{aligned}\right.
\end{equation}
has a unique solution $X\in S^2(0,T;\mathbb{R})$, for $b_1,\ \sigma_1:\mathbb{R}^2\rightarrow \mathbb{R}$, $h,\ l:\mathbb{R}\rightarrow\mathbb{R}$ bounded and Lipschitz, and $g$ defined as in Example \ref{EX2}.
\end{example}
\begin{example}\rm
Let function $l$ be defined as \eqref{exl}. In analogy to \eqref{EX3}, we obtain that the following mean-field stochastic differential equation
 \begin{equation}\label{eq 2.13}
\left\{
             \begin{array}{l}
           dX_t=\Big(\displaystyle l(X_t)+l(E[X_t])\Big)dt+(X_t+E[X_t])dW_t,~t\in[0,T],\\
X_0=x_{0}\in\mathbb{R} ,
             \end{array}
\right.
  \end{equation}
 has a unique solution\ $X\in S^2(0,T;\mathbb{R})$.
\end{example}
\section{The stochastic control problem under monotonicity condition}
\indent We consider the following general controlled mean-field stochastic differential system
\begin{equation}\label{eq 4.1}
\left\{
             \begin{array}{l}
           dX^u_t=b(t,X^u_t,\mathbb{P}_{X^u_t},u_t)dt+\sigma(t,X^u_t,\mathbb{P}_{X^u_t},u_t)dW_t,\\
X^u_0=x_{0} ,
             \end{array}
\right.
\end{equation}
where\ $b,\sigma:[0,T]\times\mathbb{R}\times\mathscr{P}_2(\mathbb{R})\times U \to \mathbb{R}$ are  deterministic functions.
An admissible control\ $u$ is an \ $\mathbb{F}$-adapted and square-integrable process, with values in a given non-empty convex subset\ $U$ of $\mathbb{R}^m$. We denote the set of admissible controls by\ ${\mathscr{U}}_{ad}$.\\
\indent The cost functional is defined as follows:
\begin{equation}\label{eq 4.2}
J(u)=E\left[\int_0^T f(t,X_{t}^{u},{\mathbb{P}}_{X_{t}^{u}},u_{t})dt+h(X_{T}^{u},{\mathbb{P}}_{X_{T}^{u}})\right],
\end{equation}
where $f:[0,T]\times\mathbb{R}\times\mathscr{P}_2(\mathbb{R})\times U \to \mathbb{R},\ h:\mathbb{R}\times\mathscr{P}_2(\mathbb{R})\to\mathbb{R}$ are  deterministic functions.\\
\indent We call \ $u^*\in{{\mathscr{U}}_{ad}}$ an optimal control, if it satisfies
 \begin{equation}\label{eq 4.3}
  J(u^*)=\inf\limits_{u\in{{\mathscr{U}}_{ad}}}J(u),
 \end{equation}
and we denote the corresponding state process by\ $X^*:=X^{u^*}$, which is the solution of equation\ (\ref{eq 4.1}) with the optimal control process $u^*$.

We give the following assumptions:
\begin{itemize}
\item[(H1)]\  Suppose that for $t\in[0,T]$, $x$, $m\in\mathbb{R}$, $u\in U$, $\xi\in L^2(\mathcal{F};\mathbb{R})$ and $\mu\in\mathscr{P}_2(\mathbb{R})$, with the notation
    $$
    \widetilde{b}(t,x,m,\mathbb{P}_{\mathring{\xi}},u):=b(t,x,\mathbb{P}_{\mathring{\xi}+m},u),\text{~for~} \mathring{\xi}:=\xi-E[\xi],
    $$
    and
    $$
    \widetilde{b}(t,x,m,\mu,u)= \widetilde{b}_0(t,x,\mu,u)+ \widetilde{b}_1(t,m,\mu),
    $$
    where $\widetilde{b}_0:[0,T]\times\mathbb{R}\times\mathscr{P}_2(\mathbb{R})\times U \to \mathbb{R}$ and $\widetilde{b}_1:[0,T]\times\mathbb{R}\times\mathscr{P}_2(\mathbb{R}) \to \mathbb{R}$,  the function $\widetilde{b}$ is Borel measurable and differentiable w.r.t. $(x,~m,~\mu,~u)$, and its derivatives are continuous, with $\partial_u \widetilde{b}_0$, $\partial_\mu \widetilde{b}_0$, $\partial_\mu \widetilde{b}_1$ bounded, $\partial_m \widetilde{b}_1(t,x,\mu)\leq\alpha_3$, $\partial_x \widetilde{b}_0(t,x,\mu,u)\leq\alpha_2$, and $\partial_x \widetilde{b}$ is of polynomial growth. Moreover, $b(\cdot,0,\delta_0,\cdot)\leq C,$ $dsd\mathbb{P}$-a.e., for some constant $C>0$.

\item[(H2)] Suppose that the function $\sigma$ is Borel measurable and differentiable w.r.t. $(x,~\mu,~u)$, and the derivatives are bounded and continuous. Moreover, $\sigma(\cdot,0,\delta_0,\cdot)\in L^p_{\mathbb{F}}(0,T;\mathbb{R})$, where $p\geq2$, and $\delta_0$ is the Dirac measure at $0$.

 \item[(H3)]Suppose that functions $f$ and $h$ are Borel measurable and differentiable w.r.t. $(x,~\mu,~u)$ and $(x,~\mu)$, respectively. The derivatives w.r.t. $(x,~\mu)$ are of polynomial growth and continuous, and $\partial_u f$ is bounded and continuous.
 \end{itemize}

 \begin{remark}\text{ }\\
 \rm(i)\ Let us point out that\ $\widetilde{b}_0=g(x),\ \widetilde{b}_1:=g(m)$, for $g$ defined in Example \ref{EX2}, satisfy the above assumptions.\\
\rm(ii)\ From (H1), for all $t,\ \theta,\ X,\ X',\ u_t\in L^2(\mathcal{F};\mathbb{R}),\ \mathring{\xi}\in L^2_0(\mathcal{F};\mathbb{R})$, we get
$$
\begin{aligned}
&E\Big[\Big(\widetilde b (t,\theta,E[X], \mathbb{P}_{\mathring{\xi}}, u_t)-b (t,\theta,E[X'], \mathbb{P}_{\mathring{\xi}}, u_t)\Big)(X-X')\Big]\\
&=E\big[\big(b_1(t,E[X],\mathbb{P}_{\mathring{\xi}})-b_1(t,E[X'],\mathbb{P}_{\mathring{\xi}})\big)(X-X')\big]\\
&= \big(b_1(t,E[X],\mathbb{P}_{\mathring{\xi}})-b_1(t,E[X'],\mathbb{P}_{\mathring{\xi}})\big)(E[X]-E[X'])\\
& \leq \alpha_3 \big(E[X]-E[X']\big)^2\leq \alpha_3 E[|X-X'|^2],
\end{aligned}
$$
\qquad i.e., we have $\rm(B3)$-(\ref{A14a}).
\end{remark}
One checks easily that our coefficients $b$ and $\sigma$ satisfy the assumptions $\rm(B1)$-$\rm(B5)$.
Below we first show that the equation (\ref{eq 4.1}) has a unique solution.
 \begin{proposition}\label{th4.1}
 Under the assumptions  $\rm(H1)$-$\rm(H2)$, the mean-field stochastic differential equation (\ref{eq 4.1}) has  a unique solution $X^u\in S^p(0,T;\mathbb{R}^n)$. Moreover, there exists a constant $C$ depending on $T$ and $p\ge 2$, such that
 \begin{equation}\label{eq 4.4}
 E[\sup_{t\in[0,T]}|X^u_t|^p ]\le C(1+|x_0|^p).
 \end{equation}
 \end{proposition}
We remark that  Theorem \ref{le2.5} implies Proposition \ref{th4.1} directly.

\subsection{Pontryagin's  stochastic maximum principle}
In this section, we study the mean-field stochastic maximum principle in the case where the control domain $U$ is convex. In the following, $C$ denotes a  constant that can change from line to line.

Let us consider the following perturbation of the optimal control $u^*\in \mathscr{U}_{ad}$:
\begin{equation}\label{eq 4.5}
u_t^\theta=u^*_t+\theta v_t,\ \text{where}\ v_t=u_t-u^*_t,\ \text{and~}u\in \mathscr{U}_{ad}.
\end{equation}
We denote by $X^\theta$  the associated state process with $u^\theta\in\mathscr{U}_{ad}$.
\\
\indent In what follows we introduce by now standard notations for the computations in the frame of Pontryagin's maximum principle.  Let $\phi=\sigma,~f,~h$.\ For  $u\in\mathscr{U}_{ad}$, $t\in[0,T]$, we put $\phi^u(t):=\phi(t,X^u_t,\mathbb{P}_{X^u_t},u_t)$ and $\phi(t):=\phi(t,X_t^*,\mathbb{P}_{X^*_t},u^*_t)$, where $u^*$ is the optimal control. We put
\begin{equation}\label{eq 4.6}
\left\{\begin{array}{l}
\phi_{x}(t):=\partial_x \phi\left(t,X_{t}^{*}, \mathbb{P}_{X_{t}^{*}}, u_{t}^{*}\right); \\
\phi_{u}(t):=\partial_u \phi\left(t,X_{t}^{*}, \mathbb{P}_{X_{t}^{*}}, u_{t}^{*}\right);\\
\phi_{\mu}(t, y):=\partial_{\mu} \phi\left(t,X_{t}^{*}, \mathbb{P}_{X_{t}^{*}}, u_{t}^{*} ; y\right).
\end{array}\right.
\end{equation}
Similar to $\widetilde{b}$, we denote
\begin{equation}\label{tildeb}
\left\{\begin{array}{l}
\widetilde{b}^u(t):=\widetilde{b}(t,X^u_t,E[X^u_t],\mathbb{P}_{\mathring{X}^u_t},u_t);\\
\widetilde{b}(t):=\widetilde{b}(t,X_t^*,E[X^*_t],\mathbb{P}_{\mathring{X}^*_t},u^*_t);\\
\widetilde{b}_{x}(t):=\partial_x \widetilde{b}(t,X_t^*,E[X^*_t],\mathbb{P}_{\mathring{X}^*_t}, u_{t}^{*}); \\
\widetilde{b}_{m}(t):=\partial_m \widetilde{b}(t,X_t^*,E[X^*_t],\mathbb{P}_{\mathring{X}^*_t}, u_{t}^{*});\\
\widetilde{b}_{u}(t):=\partial_u \widetilde{b}(t,X_t^*,E[X^*_t],\mathbb{P}_{\mathring{X}^*_t}, u_{t}^{*});\\
\widetilde{b}_{\mu}(t, y):=\partial_{\mu} \widetilde{b}(t,X_t^*,E[X^*_t],\mathbb{P}_{\mathring{X}^*_t}, u_{t}^{*} ; y),~t\in[0,T].
\end{array}\right.
\end{equation}
Obviously, $\phi_x(\cdot),~\widetilde{b}_{x}(\cdot),~\widetilde{b}_{m}(\cdot),~\phi_{u}(\cdot),~\widetilde{b}_{u}(\cdot),~\phi_\mu(\cdot,y),~\widetilde{b}_{\mu}(\cdot, y),\ y\in\mathbb{R}$ are progressively measurable processes on $(\Omega,\mathcal{F},\mathbb{P})$.  Let $(\widehat{\Omega},\widehat{\mathcal{F}},\widehat{\mathbb{P}})$ be the copy of $(\Omega,\mathcal{F},\mathbb{P})$.  We consider the product space $(\Omega\times\widehat{\Omega},\mathcal{F}\otimes\widehat{\mathcal{F}},\mathbb{P}\otimes\widehat{\mathbb{P}})$,\ and denote all processes defined on\ $(\widehat{\Omega},\widehat{\mathcal{F}},\widehat{\mathbb{P}})$ with superscript ''$\widehat{~~~~}$".\ Let\ $(\widehat{u}^*,\widehat{X}^*)$ be an independent copy of $(u^*,X^*)$, such that\ $\mathbb{P}_{(X^*_t,u^*_t)}=\widehat{\mathbb{P}}_{(\widehat{X}^*_t,\widehat{u}^*_t)}$.\ Then we set:
\begin{equation}\label{eq 4.7}
\left\{\begin{array}{l}
\widehat{\phi}_{\mu}(t):=\partial_{\mu} \phi(t,X_{t}^{*}, \mathbb{P}_{X_{t}^{*}}, u_{t}^{*} ;\widehat{X}_{t}^{*});\\
\widehat{\phi}_{\mu}^{\star}(t):=\partial_{\mu} \phi(t,\widehat{X}_{t}^{*}, \mathbb{P}_{X_{t}^{*}}, \widehat{u}_{t}^{*} ; X_{t}^{*});\\
\widehat{\widetilde{b}}_{\mu}(t):=\partial_{\mu} \widetilde{b}(t,X_t^*,E[X^*_t],\mathbb{P}_{\mathring{X}^*_t}, u_{t}^{*} ; \widehat{\mathring{X}}^{\substack{*\\[-4ex]~}}_t);\\
\widehat{\widetilde{b}}^{\substack{\star\\[-4ex]~}}_{\mu}(t):=\partial_{\mu} \widetilde{b}(t,\widehat{X}_t^*,E[X^*_t],\mathbb{P}_{\mathring{X}^*_t}, \widehat{u}_{t}^{*} ; \mathring{X}^*_t)
,\ t \in[0, T].
\end{array}\right.
\end{equation}

We will prove that the cost functional \ $J(\cdot)$ is\ Fr\'{e}chet differentiable. Then we have $$J'(u^*)(u(\cdot)-u^*(\cdot))\geq 0,\text{~for all~} u\in\mathscr{U}_{ad}.$$
We give the Hamiltonian function, for all $(t,x,m,\mu,u,p,q)\in[0,T]\times\mathbb{R}\times\mathbb{R}\times\mathscr{P}_2(\mathbb{R})\times U\times\mathbb{R}\times\mathbb{R}$,
\begin{equation}\label{eq 4.8}
H(t,x,m,\mu,u,p,q):=\widetilde{b}(t,x,m,\mathring{\mu},u)p+\sigma(t,x,\mu,u)q+f(t,x,\mu,u),
\end{equation}
where $\mathring{\mu}=\mu\circ[\psi]^{-1}$ is the image measure of $\mu$ w.r.t. $\psi(y)=y-\int_{\mathbb{R}}y\mu(dy),~y\in\mathbb{R}$, and ($p,q$) is the solution of adjoint equation \eqref{eq 4.9}, i.e.,
\begin{equation}\label{eq 4.9}
\left\{\begin{array}{l}
dp_t=-\{\widetilde{b}_x(t)p_t+\widetilde{b}_m(t)E[p_t]+\widehat E[\widehat{\widetilde b}_\mu^{\substack{\star\\[-4ex]~}}(t)\widehat p_t]-\widehat E[E[\widehat{\widetilde b}_\mu^{\substack{\star\\[-4ex]~}}(t)]\widehat p_t]+\sigma_x(t)q_t+\widehat E[\widehat\sigma_\mu^\star(t)\widehat q_t]\\
~~~~~~~~~~~+f_x(t)+\widehat E[\widehat{f}_\mu^\star(t)]\}dt+q_tdW_t,\\
p_T=h_x(T)+\widehat E[\widehat h_\mu^\star(T)],~t\in[0,T].
\end{array}\right.
\end{equation}
Equation\ (\ref{eq 4.9})\ is a linear\ BSDE of mean-field type. In analogy to Theorem\ \ref{le2.5},\ we  prove that it has a unique solution\ ($p,q$)\ satisfying
$$E\Big[\sup _{t \in[0, T]}\left|p_{t}\right|^{2}+\int_{0}^{T}\left|q_{t}\right|^{2} d t\Big]<+\infty.$$
\begin{lemma}\label{le4.1}
Suppose\ $Z$ solves the following  variation equation,
\begin{equation}\label{eq 4.10}
\left\{\begin{array}{l}
\begin{aligned}
d Z_{t}=&\left\{\widetilde{b}_{x}(t) Z_{t}+\widetilde{b}_{m}(t)E[Z_{t}]+\widehat{E}[\widehat{\widetilde{b}}_{\mu}(t) \widehat{\mathring{Z}}_{t}]+ \widetilde{b}_u(t) v_t\right\} d t \\
&+\left\{\sigma_{x}(t) Z_{t}+\widehat{E}[\widehat{\sigma}_{\mu}(t) \widehat{Z}_{t}]+\sigma_u(t) v_t\right\} d W_{t}, ~t\in[0,T],\\
Z_{0}=&\ 0.
\end{aligned}
\end{array}\right.
\end{equation}
(Recall the notation $\mathring{Z}_{t}=Z_t-E[Z_t]$, $\widehat{\mathring{Z}}_{t}=\widehat{Z}_t-E[Z_t]$ ).
Then it holds,
$$
\lim_{\theta\to0}\sup_{t \in[0, T]}E\left[\left|\frac{X_t^\theta-X_t^*}{\theta}-Z_t\right|^2\right]=0.
$$
\end{lemma}
\begin{proof}[Proof.]
In analogy to Theorem \ref{le2.5}, we see that  equation\ (\ref{eq 4.10}) has a unique solution $Z\in\mathcal{S}^2(0,T;\mathbb{R})$ such that, for all  $p\in \mathbb{N}_+$,  $E[\sup\limits_{t\in[0,T]}|Z_t|^p]<\infty$.  We put $\displaystyle Y_t^\theta:=\frac{X_t^\theta-X_t^*}{\theta}-Z_t$.
Then\ $Y_0^\theta=0$, $\mathbb{P}$-a.s.,\ and\ $Y_{\cdot}^\theta$ satisfies the following\ SDE,
\begin{equation}\label{eq 4.11}
\begin{aligned}
dY_t^\theta=&\frac{1}{\theta}\Big[\widetilde{b}(t,X_t^*+\theta(Y_t^\theta+Z_t),E[X_t^*+\theta(Y_t^\theta+Z_t)],\mathbb{P}_{\mathring{X}_t^*+\theta(\mathring{Y}_t^\theta+\mathring{Z}_t)}
,u_t^*+\theta v_t )-\widetilde{b}(t)\Big]dt
\\
&-\left[\widetilde{b}_x(t)Z_t+\widetilde{b}_m(t)E[Z_t]+\widehat{E}[\widehat
{\widetilde{b}}_\mu(t)\widehat{\mathring{Z}}_t]+\widetilde{b}_u(t)v_t\right]dt
\\
&+\frac{1}{\theta}[\sigma(t,X_t^*+\theta(Y_t^\theta+Z_t),\mathbb{P}_{X_t^*+\theta(Y_t^\theta+Z_t)}
,u_t^*+\theta v_t )-\sigma(t)]dW_t
\\
&-\left[\sigma_x(t)Z_t+\widehat E[\widehat\sigma_\mu(t)\widehat Z_t]+\sigma_u(t)v_t\right]dW_t,~t\in[0,T].
\end{aligned}
\end{equation}
Since
$$\frac{d}{d\lambda}\widetilde{b}(\cdot,\cdot,\cdot,\mathbb{P}_{\mathring{X}_t^*+\lambda\theta(\mathring{Y}_t^\theta+\mathring{Z}_t)},\cdot)
=\widehat E\Big[\widetilde{b}_\mu(\cdot,\cdot,\cdot,\mathbb{P}_{\mathring{X}_t^*+\lambda\theta(\mathring{Y}_t^\theta+\mathring{Z}_t)},\cdot~;\widehat{\mathring{X}}_t^{\substack{*\\[-4ex]~}}+\lambda\theta(\widehat{\mathring{Y}}^{\substack{\theta\\[-4ex]~}}_t+\widehat{\mathring{Z}}_t))(\widehat{\mathring{Y}}_t^{\substack{\theta\\[-4ex]~}}+\widehat{\mathring{Z}}_t)\Big]\theta,
$$
 we have
\begin{equation}\label{eq4.12}
\begin{aligned}
&\frac{1}{\theta}\Big( \widetilde{b}(t,X_t^*+\theta(Y_t^\theta+Z_t),E[X_t^*+\theta(Y_t^\theta+Z_t)],\mathbb{P}_{\mathring{X}_t^*+\theta(\mathring{Y}_t^\theta+\mathring{Z}_t)}
,u_t^*+\theta v_t )-\widetilde{b}(t)\Big)
\\
=&\int_0^1 \widetilde{b}_x(t,X_t^{\lambda,\theta},E[X_t^{\lambda,\theta}],\mathbb{P}_{\mathring{X}_t^{\lambda,\theta}}
,u_t^{\lambda,\theta} )(Y_t^\theta+Z_t)d\lambda
\\
&+\int_0^1 \widetilde{b}_m(t,X_t^{\lambda,\theta},E[X_t^{\lambda,\theta}],\mathbb{P}_{\mathring{X}_t^{\lambda,\theta}}
,u_t^{\lambda,\theta} )E[Y_t^\theta+Z_t]d\lambda
\\
&+\int_0^1 \widehat E\big[ \widetilde{b}_\mu(t,X_t^{\lambda,\theta},E[X_t^{\lambda,\theta}],\mathbb{P}_{\mathring{X}_t^{\lambda,\theta}}
,u_t^{\lambda,\theta};\widehat{\mathring{X}}_t^{\substack{\lambda,\theta\\[-4ex]~}})(\widehat{\mathring{Y}}_t^{\substack{\theta\\[-4ex]~}}+\widehat{\mathring{Z}}_t)\big]d\lambda
\\
&+\int_0^1 \widetilde{b}_u(t,X_t^{\lambda,\theta},E[X_t^{\lambda,\theta}],\mathbb{P}_{\mathring{X}_t^{\lambda,\theta}}
,u_t^{\lambda,\theta})v_td\lambda,
\end{aligned}
\end{equation}
where   $X_t^{\lambda,\theta}:=X_t^*+\lambda\theta(Y_t^\theta+Z_t)$, and   $u_t^{\lambda,\theta}=u_t^*+\lambda\theta v_t$,
 (\ref{eq 4.11})\ can be rewritten as follows:
\begin{align*}
&dY^\theta_t\\
=&\bigg\{\int_0^1 \widetilde{b}_x(t,X_t^{\lambda,\theta},E[X_t^{\lambda,\theta}],\mathbb{P}_{\mathring{X}_t^{\lambda,\theta}}
,u_t^{\lambda,\theta} )Y_t^\theta d\lambda+\int_0^1 \widetilde{b}_m(t,X_t^{\lambda,\theta},E[X_t^{\lambda,\theta}],\mathbb{P}_{\mathring{X}_t^{\lambda,\theta}}
,u_t^{\lambda,\theta} )E[Y_t^\theta] d\lambda
\\&+\int_0^1 \widehat E\big[ \widetilde{b}_\mu(t,X_t^{\lambda,\theta},E[X_t^{\lambda,\theta}],\mathbb{P}_{\mathring{X}_t^{\lambda,\theta}}
,u_t^{\lambda,\theta};\widehat{\mathring{X}}_t^{\substack{\lambda,\theta\\[-4ex]~}})\widehat{\mathring{Y}}_t^{\substack{\theta\\[-4ex]~}}\big]d\lambda
\\
&+\int_0^1 (\widetilde{b}_x(t,X_t^{\lambda,\theta},E[X_t^{\lambda,\theta}],\mathbb{P}_{\mathring{X}_t^{\lambda,\theta}}
,u_t^{\lambda,\theta} )-\widetilde{b}_x(t))Z_td\lambda\\
&+\int_0^1 (\widetilde{b}_m(t,X_t^{\lambda,\theta},E[X_t^{\lambda,\theta}],\mathbb{P}_{\mathring{X}_t^{\lambda,\theta}}
,u_t^{\lambda,\theta} )-\widetilde{b}_m(t))E[Z_t]d\lambda\\
&+\int_0^1\widehat E\big[\big( \widetilde{b}_\mu(t,X_t^{\lambda,\theta},E[X_t^{\lambda,\theta}],\mathbb{P}_{\mathring{X}_t^{\lambda,\theta}}
,u_t^{\lambda,\theta};\widehat{\mathring{X}}_t^{\substack{\lambda,\theta\\[-4ex]~}})-\widehat
{\widetilde{b}}_\mu(t)\big)\widehat{\mathring{Z}}_t\big] d\lambda
\\
&+\int_0^1(\widetilde{b}_u(t,X_t^{\lambda,\theta},E[X_t^{\lambda,\theta}],\mathbb{P}_{\mathring{X}_t^{\lambda,\theta}}
,u_t^{\lambda,\theta} )-\widetilde{b}_u(t))v_td\lambda\bigg\}dt+\bigg\{\int_0^1 \sigma_x(t,X_t^{\lambda,\theta},\mathbb{P}_{X_t^{\lambda,\theta}},u_t^{\lambda,\theta} )Y_t^\theta d\lambda\\
&+\int_0^1 \widehat E[ \sigma_\mu(t,X_t^{\lambda,\theta},\mathbb{P}_{X_t^{\lambda,\theta}},u_t^{\lambda,\theta} ;\widehat X_t^{\lambda,\theta})\widehat Y_t^\theta]d\lambda
+\int_0^1 (\sigma_x(t,X_t^{\lambda,\theta},\mathbb{P}_{X_t^{\lambda,\theta}},u_t^{\lambda,\theta} )-\sigma_x(t))Z_td\lambda\\
&+ \int_0^1\widehat E [(\sigma_\mu(t,X_t^{\lambda,\theta},\mathbb{P}_{X_t^{\lambda,\theta}},u_t^{\lambda,\theta} ;\widehat X_t^{\lambda,\theta}) -\widehat \sigma_\mu(t))\widehat Z_t]  d\lambda\\
&+\int_0^1(\sigma_u(t,X_t^{\lambda,\theta},\mathbb{P}_{X_t^{\lambda,\theta}}, u_t^{\lambda,\theta} )-\sigma_u(t))v_td\lambda\bigg\}dW_t,~t\in[0,T].
\end{align*}
Applying It\^{o}'s formula to\ $|Y_t^\theta|^2$, we have
\begin{align*}
&E\Big[|Y_t^\theta|^2\Big]\\
=&E\bigg[\int_0^t 2Y_s^\theta\int_0^1\left\{ \widetilde{b}_x(s,X_s^{\lambda,\theta},E[X_s^{\lambda,\theta}],\mathbb{P}_{\mathring{X}_s^{\lambda,\theta}}
,u_s^{\lambda,\theta} )Y_s^\theta + \widetilde{b}_m(s,X_s^{\lambda,\theta},E[X_s^{\lambda,\theta}],\mathbb{P}_{\mathring{X}_s^{\lambda,\theta}}
,u_s^{\lambda,\theta} )E[Y_s^\theta]\right.\\
&\left.~~~+ \widehat E\big[ \widetilde{b}_\mu(s,X_s^{\lambda,\theta},E[X_s^{\lambda,\theta}],\mathbb{P}_{\mathring{X}_s^{\lambda,\theta}}
,u_s^{\lambda,\theta};\widehat{\mathring{X}}_s^{\substack{\lambda,\theta\\[-4ex]~}})\widehat{\mathring{Y}}_s^{\substack{\theta\\[-4ex]~}}\big]\right\}d\lambda ds\bigg]\\
&+E\bigg[\int_0^t 2Y_s^\theta\bigg\{\int_0^1 (\widetilde{b}_x(s,X_s^{\lambda,\theta},E[X_s^{\lambda,\theta}],\mathbb{P}_{\mathring{X}_s^{\lambda,\theta}}
,u_s^{\lambda,\theta} )-\widetilde{b}_x(s))Z_sd\lambda\\
&~~~~~~~+\int_0^1 (\widetilde{b}_m(s,X_s^{\lambda,\theta},E[X_s^{\lambda,\theta}],\mathbb{P}_{\mathring{X}_s^{\lambda,\theta}}
,u_s^{\lambda,\theta} )-\widetilde{b}_m(s))E[Z_s]d\lambda\\
&~~~~~~~+\int_0^1\widehat E\big[\big( \widetilde{b}_\mu(s,X_s^{\lambda,\theta},E[X_s^{\lambda,\theta}],\mathbb{P}_{\mathring{X}_s^{\lambda,\theta}}
,u_s^{\lambda,\theta};\widehat{\mathring{X}}_s^{\substack{\lambda,\theta\\[-4ex]~}})-\widehat
{\widetilde{b}}_\mu(s)\big)\widehat{\mathring{Z}}_s\big] d\lambda\\
&~~~~~~~+\int_0^1(\widetilde{b}_u(s,X_s^{\lambda,\theta},E[X_s^{\lambda,\theta}],\mathbb{P}_{\mathring{X}_s^{\lambda,\theta}}
,u_s^{\lambda,\theta} )-\widetilde{b}_u(s))v_sd\lambda\bigg\}ds\bigg]\\
&+E\bigg[\int_0^t\bigg| \int_0^1 \big(\sigma_x(s,X_s^{\lambda,\theta},\mathbb{P}_{X_s^{\lambda,\theta}},u_s^{\lambda,\theta} )Y_s^\theta
+\widehat E[ \sigma_\mu(s,X_s^{\lambda,\theta},\mathbb{P}_{X_s^{\lambda,\theta}},u_s^{\lambda,\theta} ;\widehat X_s^{\lambda,\theta})\widehat Y_s^\theta]\big)d\lambda\\
&~~~~~~~+\int_0^1 (\sigma_x(s,X_s^{\lambda,\theta},\mathbb{P}_{X_s^{\lambda,\theta}},u_s^{\lambda,\theta} )-\sigma_x(s))Z_sd\lambda\\
&~~~~~~~+\int_0^1 \widehat E [(\sigma_\mu(s,X_s^{\lambda,\theta},\mathbb{P}_{X_s^{\lambda,\theta}},u_s^{\lambda,\theta} ;\widehat X_s^{\lambda,\theta}) -\widehat \sigma_\mu(s))\widehat Z_s] d\lambda\\
&~~~~~~~+\int_0^1(\sigma_u(s,X_s^{\lambda,\theta},\mathbb{P}_{X_s^{\lambda,\theta}}, u_s^{\lambda,\theta} )-\sigma_u(s))v_sd\lambda \bigg|^2ds\bigg],~t\in[0,T].
\end{align*}
Using the assumption (H1) for $\widetilde{b}(s,x,m,\mu,u)$ w.r.t. $x$ and to $m$, we have
\begin{align*}
&\widetilde{b}_x(s,X_s^{\lambda,\theta},E[X_s^{\lambda,\theta}],\mathbb{P}_{\mathring{X}_s^{\lambda,\theta}}
,u_s^{\lambda,\theta} )|Y_s^\theta|^2\le \alpha_2|Y_s^\theta|^2,\\
&\widetilde{b}_m(s,X_s^{\lambda,\theta},E[X_s^{\lambda,\theta}],\mathbb{P}_{\mathring{X}_s^{\lambda,\theta}}
,u_s^{\lambda,\theta} )|E[Y_s^\theta]|^2\le \alpha_3|E[Y_s^\theta]|^2.
\end{align*}
Moreover, recalling that $\widetilde{b}_{\mu}$ is bounded, we get
\begin{equation}\label{eq 4.13}
\begin{aligned}
&E\Big[\sup_{r\in[0,t]}\int_0^r Y^\theta_s\int_0^1\widehat E\big[ \widetilde{b}_\mu(s,X_s^{\lambda,\theta},E[X_s^{\lambda,\theta}],\mathbb{P}_{\mathring{X}_s^{\lambda,\theta}}
,u_s^{\lambda,\theta};\widehat{\mathring{X}}_s^{\substack{\lambda,\theta\\[-4ex]~}})\widehat{\mathring{Y}}_s^{\substack{\theta\\[-4ex]~}}\big]d\lambda ds\Big]\\
\le & CE\Big[\sup_{r\in[0,t]}\int_0^r|Y^\theta_s|\widehat E[|\widehat{Y}_s^{\theta}|]ds\Big]\le C\int_0^t E[|Y^\theta_s|^2]ds,~t\in[0,T].
\end{aligned}
\end{equation}
Consequently, in virtue of the assumptions on  $\sigma_x $ and  $\sigma_\mu $, we deduce from Young's inequality  that
\begin{equation*}
E[|Y_t^\theta|^2]\leq CE[\int_0^t|Y_s^\theta|^2ds ]+\rho(\theta),~t\in[0,T],
\end{equation*}
where
\begin{align*}
\rho(\theta):=~&CE\left[\int_0^T|Z_t|^2\int_0^1|\widetilde{b}_x(t,X_t^{\lambda,\theta},E[X_t^{\lambda,\theta}],\mathbb{P}_{\mathring{X}_t^{\lambda,\theta}}
,u_t^{\lambda,\theta} )-\widetilde{b}_x(t)|^2d\lambda dt \right]
\\
&+CE\left[\int_0^T|E[Z_t]|^2\int_0^1|\widetilde{b}_m(t,X_t^{\lambda,\theta},E[X_t^{\lambda,\theta}],\mathbb{P}_{\mathring{X}_t^{\lambda,\theta}}
,u_t^{\lambda,\theta} )-\widetilde{b}_m(t)|^2d\lambda dt \right]
\\
&+CE\left[\int_0^T\int_0^1\big|\widehat E\big[\big( \widetilde{b}_\mu(t,X_t^{\lambda,\theta},E[X_t^{\lambda,\theta}],\mathbb{P}_{\mathring{X}_t^{\lambda,\theta}}
,u_t^{\lambda,\theta};\widehat{\mathring{X}}_t^{\substack{\lambda,\theta\\[-4ex]~}})-\widehat
{\widetilde{b}}_\mu(t)\big)\widehat{\mathring{Z}}_t\big]\big|^2d\lambda dt \right]
\\
&+CE\left[\int_0^T |v_t|^2\int_0^1|\widetilde{b}_u(t,X_t^{\lambda,\theta},E[X_t^{\lambda,\theta}],\mathbb{P}_{\mathring{X}_t^{\lambda,\theta}}
,u_t^{\lambda,\theta} )-\widetilde{b}_u(t)|^2d\lambda dt \right]
\\
&+CE\left[\int_0^T |Z_t|^2\int_0^1|\sigma_x (t,X_t^{\lambda,\theta},\mathbb{P}_{X_t^{\lambda,\theta}},u_t^{\lambda,\theta})-
\sigma_x(t)|^2d\lambda dt \right]
\\
&+CE\left[\int_0^T\int_0^1|\widehat E [(\sigma_\mu(t,X_t^{\lambda,\theta},\mathbb{P}_{X_t^{\lambda,\theta}},u_t^{\lambda,\theta} ;\widehat X_t^{\lambda,\theta}) -\widehat \sigma_\mu(t))\widehat Z_t]|^2d\lambda dt \right]
\\
&+CE\left[\int_0^T |v_t|^2\int_0^1|\sigma_u (t,X_t^{\lambda,\theta},\mathbb{P}_{X_t^{\lambda,\theta}},u_t^{\lambda,\theta})-
\sigma_u(t)|^2d\lambda dt \right].
\end{align*}
(Recall that $\widetilde{b}(t,x,m,\mu,u)=\widetilde{b}_0(t,x,\mu,u)+\widetilde{b}_1(t,m,\mu)$).
From the the polynomial growth condition on $b_x$  w.r.t. $(x,m,\mu,u)$, we see that
\begin{equation*}
\begin{aligned}
&\int_0^T E\left[|Z_t|^4\int_0^1|\widetilde{b}_x(t,X_t^{\lambda,\theta},E[X_t^{\lambda,\theta}],\mathbb{P}_{\mathring{X}_t^{\lambda,\theta}}
,u_t^{\lambda,\theta} )-\widetilde{b}_x(t,X_t^{*},E[X_t^{*}],\mathbb{P}_{\mathring{X}_t^{*}}
,u_t^{*} )|^4d\lambda\right]dt\\
\le&CE\Big[\Big(1+\int_0^T (|X_t^{\theta}|^{4p}+| X^*_t|^{4p})dt\Big)\sup_{t\in[0,T]}|Z_t|^4\Big]\\
\le&C(1+|x_0|^{4p}) <+\infty, \text{~for all~} \theta\in[0,1].
\end{aligned}
\end{equation*}
Hence, from the continuity of  $\widetilde{b}_x$ and Lebesgue's convergence theorem, we have
$$
E\left[\int_0^T|Z_t|^2\int_0^1|\widetilde{b}_x(t,X_t^{\lambda,\theta},E[X_t^{\lambda,\theta}],\mathbb{P}_{\mathring{X}_t^{\lambda,\theta}}
,u_t^{\lambda,\theta} )-\widetilde{b}_x(t)|^2d\lambda dt \right]\to 0,\ \theta\to 0.
$$
Similarly, we can prove that the other terms of $\rho(\theta)$ tend to 0. Concerning $$E\left[\int_0^T |v_t|^2\int_0^1|\widetilde{b}_u(t,X_t^{\lambda,\theta},E[X_t^{\lambda,\theta}],\mathbb{P}_{\mathring{X}_t^{\lambda,\theta}}
,u_t^{\lambda,\theta} )-\widetilde{b}_u(t)|^2d\lambda dt \right]$$ and also the corresponding term with $\sigma_u$, recall that $\widetilde{b}_u$ and $\sigma_u$ are continuous and bounded, while $|v|^2$ belongs to $L^1_{\mathbb{F}}(0,T;\mathbb{R})$. This allows to apply the dominated convergence theorem to deduce their convergence to zero, as $\theta\rightarrow0$.  Consequently,  $\rho(\theta)\rightarrow 0$, as $\theta\rightarrow 0$.

By using  Gronwall's inequality we obtain $$\sup_{t\in[0,T]}E[|Y_t^\theta|^2]\leq K\rho(\theta)\rightarrow 0,\ \theta\rightarrow 0 .$$
This proves the results.
\end{proof}
\begin{lemma}\label{le4.2}
The Fr\'{e}chet derivative of the cost functional is given by
\begin{equation}\label{eq 4.14}
\begin{aligned}
\left.\frac{d}{d\theta}J(u^*+\theta v)\right|_{\theta=0}=&E\left[\int_0^T(f_x(t)Z_t+\widehat E[\widehat f_\mu(t)\widehat{Z}_t]+f_u(t)v_t )dt \right]
\\
&+E\big[h_x(T)Z_T+\widehat E[\widehat h_\mu(T)\widehat{Z}_T] \big].
\end{aligned}
\end{equation}
\end{lemma}

\begin{proof}[Proof.]
Recalling the definition of cost functional, we have
$$
\frac{d}{d\theta}J(u^*+\theta v)\bigg|_{\theta=0}=\frac{d}{d\theta}E\left[\int_0^Tf(t,X_t^\theta,\mathbb{P}_{X_t^{\theta}},u^\theta_t)dt+h(X_T^\theta,\mathbb{P}_{X_T^{\theta}})\right]\bigg|_{\theta=0},$$
where,
\begin{equation*}
\begin{aligned}
&\left.\frac{d}{d\theta}E[h(X_T^\theta,\mathbb{P}_{X_T^{\theta}})]\right|_{\theta=0}=\lim_{\theta\to 0}E\left[\frac{h(X_T^\theta,\mathbb{P}_{X_T^{\theta}})-h(X_T^*,\mathbb{P}_{X_T^*})}{\theta} \right]
\\
=&\lim_{\theta\to 0}E\left[\int_0^1h_x(X_T^*+\lambda(X_T^\theta-X_T^*),\mathbb{P}_{{X_T^*}+\lambda(X_T^\theta-X_T^*) })\frac{(X_T^\theta-X_T^*) }{\theta}d\lambda \right]
\\
&+\lim_{\theta\to 0}E\left[\!\int_0^1\!\widehat E\left[h_\mu(X_T^*\!+\!\lambda(X_T^\theta-X_T^*),\mathbb{P}_{{X_T^*}+\lambda(X_T^\theta-X_T^*) };\widehat X_T^*\!+\!\lambda(\widehat  X_T^\theta\!-\!\widehat X_T^* ))\frac{(\widehat{X}_T^\theta-\widehat{X}_T^*) }{\theta}\!\right]\!d\lambda\! \right].
\end{aligned}
\end{equation*}
Since, due to Lemma \ref{le4.1},  $\displaystyle Y_t^\theta:=\frac{X_t^\theta-X_t^*}{\theta}-Z_t\rightarrow0$ in $L^2$,
 as $\theta\to0$, uniformly w.r.t. $t\in[0,T]$, we get thanks to the dominated convergence theorem
$$E\left[\int _0^1 h_x\left(X^*_T+\lambda\left(X_T^\theta-X^*_T\right),
\mathbb{P}_{{X_T^*}+\lambda(X_T^\theta-X_T^*) }\right)\frac{X_T^\theta-X^*_T}{\theta}d\lambda\right]\to E\left[h_x\left(X^*_T,\mathbb{P}_{X^*_T}\right)Z_T\right].$$
Similarly, as $\theta\to0 $, we have,
\begin{equation*}
\begin{aligned}
&E\left[\int _0^1\widehat{E} \left[h_\mu\left(X_T^*+\lambda(X_T^\theta-X_T^*),
\mathbb{P}_{{X_T^*}+\lambda(X_T^\theta-X_T^*)};\widehat{X}_T^*+\lambda(\widehat X_T^\theta-\widehat X_T^*)\right)\frac{\widehat{X}_T^\theta-\widehat{X}^*_T}{\theta}\right]d\lambda\right]\\
&\to E\left[\widehat{E}[\widehat h_\mu(T)\widehat{Z}_T]\right],
\end{aligned}
\end{equation*}
and also
$$\left.\frac{d}{d\theta}E\left[\int_0^Tf(X_t^\theta,\mathbb{P}_{X_t^\theta},u_t^\theta)dt\right]\right|_{\theta=0}
=E\left[\int_0^T(f_x(t)Z_t+\widehat E[\widehat f_\mu(t)\widehat Z_t]+ f_u(t)v_t)dt \right].
$$
The proof is complete now.
\end{proof}
The  dual relationship between adjoint process and the variational process is given by the following lemma. It is obtained by a straight-forward computation of $E[p_tZ_t]$ by using It\^{o}'s formula.
\begin{lemma}\label{le4.3}
Suppose\ ($p,q$)\ are the adjoint processes satisfying equation  (\ref{eq 4.9}), and $Z$ is the variational process satisfying equation (\ref{eq 4.10}). Then we have,
$$
E[p_TZ_T]=E\left[\int_0^T(p_t\widetilde{b}_u(t)v_t-Z_tf_x(t)-\widehat E[\widehat f_\mu(t)\widehat Z_t] +q_t\sigma_u(t)v_t)dt\right].
$$
\end{lemma}
Thus, from Lemma \ref{le4.1}, Lemma \ref{le4.2} and \ $ E[p_TZ_T]=E[ h_x(T)Z_T+\widehat E[\widehat h_\mu(T)\widehat Z_T]]$,\
we  get the following corollary.
\begin{corollary}\label{cor 4.1}
The Fr\'{e}chet derivative of the cost functional can be rewritten as
\begin{equation*}
\begin{aligned}
\frac{d}{d\theta}J(u^*+\theta v)\bigg|_{\theta=0}
=&E\left[\int_0^T(f_u(t)v_t+p_t\widetilde{b}_u(t)v_t+q_t\sigma_u(t)v_t )dt\right]
\\
=&E\left[ \int_0^T\frac{d}{du}H( t,X_t^*,E[X_t^*],\mathbb{P}_{X_t^*},u_t^*,p_t,q_t )v_tdt\right].
\end{aligned}
\end{equation*}
\end{corollary}
Now we are able to state the main theorem in this section.
\begin{theorem}\label{th4.2}
Under the assumptions  $\rm(H1)$-$\rm(H3)$, if $(X^*,u^*)$ is the optimal pair for our  stochastic control problem, then  the following inequality is satisfied,
\begin{equation}\label{4.11'}
H_u(t,X_t^*,E[X_t^*],\mathbb{P}_{X_t^*},u_t^*,p_t,q_t)(u-u_t^*)\geq0,~dtd\mathbb{P} \text{-a.e.,}~u\in U,
\end{equation}
where  $(p,q)$  is the unique solution of the adjoint BSDE (\ref{eq 4.9}).
\end{theorem}
\begin{proof}[Proof.]
Since\ $u^*$ is an optimal control, we have
$$
\left.0\leq\frac{d}{d\theta}J(u^*+\theta(u-u^*))\right|_{\theta=0}=E\left[\int_0^T\frac{d}{du}
H(t,X_t^*,E[X_t^*],\mathbb{P}_{X_t^*},u_t^*,p_t,q_t)(u_t-u_t^*)dt\right].
$$
The proof is complete.
\end{proof}
\subsection{Sufficient conditions}
Let us  give now the sufficient conditions for the optimal control problem.
To this end, we add the following assumption.
\begin{itemize}
\item[(H4)]The function $h(x, \mu)$ is convex in $(x, \mu)$,  $\widetilde{b}(t,x,m,\mathring{\mu},u)$ is convex in $\mathring{\mu}$, the functions $\sigma(t,x,\mu,u)$ and $f(t,x,\mu,u)$ are convex in $\mu$, as well as $H(t,x,m,\mu,p_t,q_t,u)$ is convex in $(x,m,u)$.
\end{itemize}
\begin{theorem}
Assume the conditions $\rm(H1)\text{-}(H4)$ are satisfied and let $u^*\in {\mathscr{U}}_{ad}$ with the associated state process $X^*_t$ be given, and let $(p, q)$ be the solution to the adjoint equation (\ref{eq 4.9}). Then, if $u^*$ satisfies the necessary condition  \eqref{4.11'} in Theorem \ref{th4.2}, i.e.,
\begin{equation*}
H_u(t,X_t^*,E[X_t^*],\mathbb{P}_{X_t^*},u_t^*,p_t,q_t)(u-u_t^*)\geq0,~dtd\mathbb{P} \text{-a.e.,}~u\in U,
\end{equation*}
\end{theorem}
 $u^*$ is an optimal control.
\begin{proof}[Proof.]
For all $u\in {\mathscr{U}}_{ad}$, we have
\begin{equation*}
J(u^*)-J(u)= E\left[h(X^*_T,\mathbb{P}_{X_T^*})- h(X^u_T,\mathbb{P}_{X_T^u}) \right]+E\left[\int_0^T \left(f(t,X^*_t,\mathbb{P}_{X_t^*},u^*_t)- f(t,X^u_t,\mathbb{P}_{X_t^u},u_t)\right)dt\right].
\end{equation*}
Since $h$ is convex in $(x,\mu)$,
\begin{equation*}
h(X^*_T,\mathbb{P}_{X_T^*})- h(X^u_T,\mathbb{P}_{X_T^u})\leq h_x(X^*_T,\mathbb{P}_{X_T^*})(X_T^*-X_T^u)+\widehat{E}\left[\partial_{\mu}h(X^*_T,\mathbb{P}_{X_T^*};\widehat{X}_T^*)(\widehat{X}_T^*-\widehat{X}_T^u)   \right].
\end{equation*}
Then,
\begin{equation}\label{eq4.13}
\begin{aligned}
J(u^*)-J(u)\leq &E \Big[ h_x(X^*_T,\mathbb{P}_{X_T^*})(X_T^*-X_T^u)+\widehat{E}[\partial_{\mu}h(X^*_T,\mathbb{P}_{X_T^*};\widehat{X}_T^*)(\widehat{X}_T^*-\widehat{X}_T^u)   ]\Big]\\
&+E\left[\int_0^T \left(f(t,X^*_t,\mathbb{P}_{X_t^*},u^*_t)- f(t,X^u_t,\mathbb{P}_{X_t^u},u_t)\right)dt\right].
\end{aligned}
\end{equation}
On the other hand, by applying It\^{o}'s formula, we get
\begin{equation*}
\begin{aligned}
E\left[p_T(X_T^*-X_T^u)\right]
=&E\left[\int_0^T \left(H(t,X_t^*,E[X_t^*],\mathbb{P}_{X_t^*},u_t^*,p_t,q_t)-H(t,X_t^u,E[X_t^u],\mathbb{P}_{X_t^u},u_t,p_t,q_t)\right)dt\right]\\
&-E\left[\int_0^T \left(f(t,X_t^*,\mathbb{P}_{X_t^*},u_t^*)-f(t,X_t^u,\mathbb{P}_{X_t^u},u_t)\right)dt\right]\\
&-E\Big[\int_0^T \Big( H_x(t,X_t^*,E[X_t^*],\mathbb{P}_{X_t^*},u_t^*,p_t,q_t)(X_t^*-X_t^u)\\
&~~~~~~+H_m(t,X_t^*,E[X_t^*],\mathbb{P}_{X_t^*},u_t^*,p_t,q_t)E[X_t^*-X_t^u]\\
&~~~~~~+\widehat{E}[\widehat{\widetilde{b}}_{\mu}(t)p_t(\widehat{\mathring{X}}_t^{\substack{*\\[-4ex]~}}-\widehat{\mathring{X}}_t^{\substack{u\\[-4ex]~}})+(\widehat{\sigma}_\mu(t)q_t+\widehat{f}_\mu(t))(\widehat{X}_t^*-\widehat{X}_t^u)]\Big)dt\Big].
\end{aligned}
\end{equation*}
From (\ref{eq4.13}) we have
\begin{equation}\label{5.13}
\begin{aligned}
J(u^*)-J(u)
\leq&E\left[\int_0^T \left(H(X_t^*,E[X_t^*],\mathbb{P}_{X_t^*},u_t^*,p_t,q_t)-H(X_t^u,E[X_t^u],\mathbb{P}_{X_t^u},u_t,p_t,q_t)\right)dt\right]\\
&-E\Big[\int_0^T \Big( H_x(t,X_t^*,E[X_t^*],\mathbb{P}_{X_t^*},u_t^*,p_t,q_t)(X_t^*-X_t^u)\\
&~~~~~~+H_m(t,X_t^*,E[X_t^*],\mathbb{P}_{X_t^*},u_t^*,p_t,q_t)E[X_t^*-X_t^u]\\
&~~~~~~+\widehat{E}[\widehat{\widetilde{b}}_{\mu}(t)p_t(\widehat{\mathring{X}}_t^{\substack{*\\[-4ex]~}}-\widehat{\mathring{X}}_t^{\substack{u\\[-4ex]~}})+(\widehat{\sigma}_\mu(t)q_t+\widehat{f}_\mu(t))(\widehat{X}_t^*-\widehat{X}_t^u)]\Big)dt\Big].
\end{aligned}
\end{equation}
Recalling the convexity of $H$, we have
\begin{equation*}
\begin{aligned}
&H(t,X_t^*,E[X_t^*],\mathbb{P}_{X_t^*},u_t^*,p_t,q_t)-H(t,X_t^u,E[X_t^u],\mathbb{P}_{X_t^u},u_t,p_t,q_t)\\
\leq & H_x(t,X_t^*,E[X_t^*],\mathbb{P}_{X_t^*},u_t^*,p_t,q_t)(X_t^*-X_t^u)+H_m(t,X_t^*,E[X_t^*],\mathbb{P}_{X_t^*},u_t^*,p_t,q_t)E[X_t^*-X_t^u]\\
&+H_u(t,X_t^*,E[X_t^*],\mathbb{P}_{X_t^*},u_t^*,p_t,q_t)(u_t^*-u_t)\\
&+\widehat{E}[\widehat{\widetilde{b}}_{\mu}(t)p_t(\widehat{\mathring{X}}_t^{\substack{*\\[-4ex]~}}-\widehat{\mathring{X}}_t^{\substack{u\\[-4ex]~}})+(\widehat{\sigma}_\mu(t)q_t+\widehat{f}_\mu(t))(\widehat{X}_t^*-\widehat{X}_t^u)].
\end{aligned}
\end{equation*}
We substitute the above inequality in \eqref{5.13}, and so we obtain
\begin{equation*}
J(u^*)-J(u)
\leq E\left[\int_0^T H_u(t,X_t^*,\mathbb{P}_{X_t^*},u_t^*,p_t,q_t)(u_t^*-u_t) dt\right].
\end{equation*}
From our condition \eqref{4.11'} the proof is complete.
\end{proof}

\end{document}